\documentclass[12pt]{llncs}

\usepackage{amssymb}
\usepackage [english]{babel}
\usepackage[utf8]{inputenc}
\usepackage[T1]{fontenc}

\usepackage{amssymb}
\usepackage{amsmath}
\usepackage{amsfonts}

\usepackage{graphicx}
\usepackage{url}

\usepackage[usenames]{color}

\usepackage{algorithm}
\usepackage{algpseudocode}

\usepackage{parskip}
\usepackage{makecell}
\usepackage{latexsym} 
\usepackage{wasysym}
\usepackage{stmaryrd}

\usepackage{xifthen}

  \newtheorem{assumption}{Assumption}

  \usepackage{geometry}
\newcolumntype{P}[1]{>{\raggedleft\arraybackslash}p{#1}}
 \usepackage{lineno}

\newcommand{\genref}[2]{#2\,\ref{#1}}
\newcommand{\lemref}[1]{\genref{#1}{Lemma}}
\newcommand{\secref}[1]{Section\,\ref{#1}}

\numberwithin{theorem}{section}
\numberwithin{proposition}{section}
\numberwithin{lemma}{section}
\numberwithin{remark}{section}
\numberwithin{corollary}{section}
\numberwithin{assume}{section}
\numberwithin{equation}{section}
    \urldef{\mailsa}\path|ulrich.langer@ricam.oeaw.ac.at|
    \urldef{\mailsd}\path|ioannis.toulopoulos@ricam.oeaw.ac.at|
    \urldef{\mailse}\path|christoph.hofer@jku.at|
    \newcommand{\keywords}[1]{\par\addvspace\baselineskip  
     \noindent\keywordname\enspace\ignorespaces#1}

\newcommand{\p}[1]{\hat{#1}}
\newcommand{\iSet}{{\mathcal{I}}}
\newcommand{\sMP}[1][]{\ifthenelse{\isempty{#1}}{^{(k)}}{^{(#1)}}}

\newcommand{\g}[1]{#1}

 \newcommand{\MatTwo}[4]
  {\begin{bmatrix}
    #1 & #2\\
    #3 & #4
  \end{bmatrix}}
  
  \newcommand{\VecTwo}[2]
  {\begin{bmatrix}
    #1 \\
    #2 
  \end{bmatrix}}

\begin{document}
\mainmatter

\title{Robust Preconditioning for Space-Time Isogeometric Analysis of Parabolic Evolution Problems}
\titlerunning{Robust Preconditioning for Space-Time IGA of Parabolic Problems.}

\author{
Christoph Hofer$^1$
\and     Ulrich Langer$^{1,2,3}$ 
\and  Martin Neum\"uller$^{2}$
}
\authorrunning{C. Hofer,  U. Langer, M. Neum\"uller }

\institute{ 
$1$ Doctoral Program ``Computational Mathematics''\\ Johannes Kepler University,\\[1ex]
$2$ Institute of Computational Mathematics,\\ Johannes Kepler University,\\[1ex]
	    $3$ Johann Radon Institute for Computational and Applied Mathematics,\\
				Austrian Academy of Sciences\\[1ex]
				Altenbergerstr. 69, A-4040 Linz, Austria,\\[1ex]
christoph.hofer@dk-compmath.jku.at\\
ulanger@numa.uni-linz.ac.at\\
neumueller@numa.uni-linz.ac.at\\
 }

\noindent
\maketitle
\pagestyle{myheadings}
\thispagestyle{plain}

\begin{abstract}
 We propose and investigate  new
robust preconditioners for space-time Isogeometric Analysis of parabolic evolution problems. 
These preconditioners are based
on a time parallel multigrid method. 
We consider a decomposition of 
the space-time cylinder
into time-slabs
 which are coupled via a discontinuous Galerkin technique.
 The time-slabs provide the structure for the time-parallel 
multigrid solver.
The most important part of the  multigrid method is the smoother. We utilize the special structure of the involved operator to decouple its application into several spatial problems by means of generalized eigenvalue or Schur decompositions. 
Some of these problems have a symmetric saddle point structure, for which we present robust preconditions. Finally, we present numerical experiments confirming 
 the robustness of our space-time IgA solver.
\end{abstract}
\section{Introduction}
\label{sec:introduction}
 Time-dependent Partial Differential Equations (PDEs) 
of parabolic type
play an important role in the simulation of various physical processes, 
like 
 heat conduction, diffusion, and 
2d eddy-current problems in electromagnetics.
They are often 
given as initial-boundary value problems (IBVP).
The discretization of such problems is usually 
performed 
either
by first discretizing in time by a time-stepping method and then in space by, e.g., finite elements or vice versa. The former method is often denoted as Rothe's method  \cite{HLN:Lang:2000a} and the latter one vertical method of lines \cite{HLN:Thomme:2006a}. Both of the two approaches are sequential in time. In order to treat such problems on massively parallel computers, different approaches are required to overcome the sequential structure. 
There exist 
 various techniques
for parallelization in time. 
We refer to \cite{HLN:Gander:2015a} for an overview of time-parallel methods.

In the current work, we focus on space-time methods.
More precisely, we consider the
time  
as just
another variable, say $x_{d+1}$, where $x_1,\ldots,x_d$ are the $d$-dimensional spatial variables. The derivative in time direction is then viewed as a strong convection term in the direction $x_{d+1}$. In order to provide a stable discretization, we use stabilization techniques developed for convection dominated elliptic convection-diffusion problems, see, e.g., \cite{HLN:Stynes:2005a}. To be more precise, we consider the Streamline-Upwind Petrov-Galerkin (SUPG) method, introduced in \cite{HLN:HughesBrooks:1982a}.
We consider the linear parabolic IBVP,
find $u : \overline{Q} \rightarrow \mathbb{R}$ such that
\begin{equation}
\label{HLN:Eqn:ParabolicIBVP-CF}
\partial_{t} u - \Delta u = f \; \mbox{in}\; Q, \; u=0 \; \mbox{on}\; \Sigma, \; \mbox{and}\; 
u=u_0 \; \mbox{on}\; \overline{\Sigma}_0,
\end{equation}
as a typical parabolic model problem posed 
in the space-time cylinder 
$ \overline{Q} = \overline{\Omega} \times \overline{J}
= \overline{\Omega} \times [0,T] 
= Q \cup \Sigma \cup \overline{\Sigma}_0 \cup \overline{\Sigma}_T$, 
where $\partial_{t}$ denotes the partial time derivative, $\Delta$ 
is the Laplace operator, $f$ is a given source function, $u_0$ are 
the given initial data, 
$T$ is the final time, 
$J= (0,T)$ is the time interval,
$Q = \Omega \times (0,T)$, $\Sigma = \partial \Omega \times (0,T)$,
$\Sigma_0 := \Omega \times \{0\}$, $\Sigma_T := \Omega \times \{T\}$, 
and 
$\Omega \subset \mathbb{R}^{d}$ $(d = 1,2,3)$ denotes the spatial 
computational domain with the boundary $\partial \Omega$.
In \cite{HLN:LangerMooreNeumueller:2016a}, a time-upwind test functions were used to construct a stable single-patch discretization scheme in the Isogeometric Analysis (IgA) framework. This approach was extended in \cite{HLN:HoferLangerNeumuellerToulopoulos:2017a} to multiple patches in time, where each space-time patch $Q_n$ is given as space-time-slab
$Q_n=\Omega\times(t_{n-1},t_n)$
corresponding to a decomposition
$0=t_{0} < t_1 < \ldots  < t_{N}=T$
of the time interval $[0,T]$.
A discontinuous Galerkin (dG) technique was used for coupling  the space-time-slabs 
in an appropriate way.
Finally, the resulting huge 
linear system $\mathbf L_h \mathbf u_h = \mathbf f_h$ is solved by 
the time-parallel
multigrid (MG) method introduced in \cite{HLN:GanderNeumueller:2016a}.
The main new contributions of this paper are the smoothers 
that finally yield a robust multigrid solver and preconditioner
for the GMRES solver, respectively.

IgA is a powerful methodology for discretizing PDEs. It was first introduced in \cite{HLN:HughesCottrellBazilevs:2005a} and its  advantages have  been highlighted  in many publications,  
see, e.g., the monograph   \cite{HLN:CotrellHughesBazilevs:2009a}, 
the survey paper \cite{HLN:BeiraodaVeigaBuffaSangalliVazquez:2014a}
and the references  therein. The main idea is to use that same smooth higher order splines for 
both
representing the computational domain 
and approximating the solution of the PDE or the PDE system.
The most common choices are B-Splines, Non-Uniform Rational B-Splines (NURBS), T-Splines, Truncated Hierarchical B-Splines (THB-Splines), etc., see, e.g., \cite{HLN:GiannelliJuettlerSpeleers:2012a}, \cite{HLN:GiannelliJuettlerSpeleers:2014a} and  \cite{HLN:BazilevsCaloCottrellEvans:2010a}.
One of the strengths of IgA is the capability of creating high-order spline spaces, while keeping the number of degrees of freedom quite small. 

The purpose of this paper is to investigate the efficient realization of the
time-parallel
MG method mentioned above. The special time-multipatch dG structure 
of the discretization leads to a block-bidiagonal matrix ${\mathbf{L}_h}= \mbox{blockbidiag}(-{\mathbf{B}_n},{\mathbf{A}_n})$, 
where the block-diagonal matrices ${\mathbf{A}_n}$, $i=1,\ldots,N$, 
and the block-subdiagonal matrices ${\mathbf{B}_n}$, $i=2,\ldots,N$, 
have tensor product representations. The most costly part of the MG method is the application of the smoother, which is of  (inexact) damped block Jacobi type, i.e,
\begin{align*}
	 u_h^{k+1} = u_h^{k} + \omega \mathbf{D}_h^{-1} \left[ f_h - \mathbf{L}_h u_h^k \right] \quad \text{for } k=1,2,\ldots.
\end{align*}
The block diagonal matrix $\mathbf{D}_h$ is formed by the diagonal blocks of $\mathbf L_h$, i.e., by $\mathbf A_n$. This paper investigates the efficient application of $\mathbf A_n^{-1}$ by utilizing its tensor product structure. 
We use ideas from \cite{HLN:Tani:2017a} and \cite{HLN:SangalliTani:2016a} to perform a decomposition of $\mathbf A_n$ into a series of spatial problems, for which we investigate robust block preconditions. These preconditioners are constructed by means of operator interpolation, see, e.g., \cite{HLN:Zulehner:2011a}, \cite{HLN:BerghLofstrom:1976a} and \cite{HLN:AdamsFournier:2003a}. Moreover, their application can be further accelerated by 
using domain decomposition or multigrid approaches 
in connection with parallelization in space.

The remainder of the paper is organized as follows. 
In Section~\ref{sec:prelim}, we rephrase basic definitions and the stable space-time dG-IgA variational formulation.
Section~\ref{sec:solvers} is devoted to the construction of efficient 
 smoothers 
used in the 
time-parallel
multigrid solver respectively preconditioner.
Numerical experiments confirming the theoretical results are presented in
Section~\ref{sec:numerics}. 
Finally, we draw some conclusions in Section~\ref{sec:Conclusions}.

\section{Preliminaries}
\label{sec:prelim}

In this section, we introduce the IgA concept,
recall some important definitions, 
and state the space-time variational IgA scheme
derived and analysed in \cite{HLN:HoferLangerNeumuellerToulopoulos:2017a}. 
For a more detailed discussion of IgA, we refer to  \cite{HLN:CotrellHughesBazilevs:2009a} 
and \cite{HLN:BeiraodaVeigaBuffaSangalliVazquez:2014a}.
We follow the notation used in \cite{HLN:HoferLangerNeumuellerToulopoulos:2017a}.

\subsection{Isogeometric Analysis}
\label{sec:iga}
Let $\p{\Omega}:=(0,1)^d$, be the d-dimensional unit cube, which we refer to as the \emph{parameter domain}. 
Let $p_\iota$ and $N_\iota,\iota\in\{1,\ldots,d\}$, denote 
the degree
and the number of basis functions in $x_\iota$-direction. Moreover, let $\Xi_\iota = \{\xi_1=0,\xi_2,\ldots,\xi_{n_\iota}=1\}$, $n_\iota=N_\iota-p_\iota-1$, be a partition of $[0,1]$, called \emph{knot vector}. With this ingredients we are able to define the B-Spline basis $\p{N}_{i,p}$, $i\in\{1,\ldots,N_\iota\}$ on $[0,1]$ via Cox-De Boor's algorithm, cf. \cite{HLN:CotrellHughesBazilevs:2009a}. The generalization to $\p{\Omega}$ is realized by considering a tensor product, again denoted by $\p{N}_{i,p}$, where $i=(i_1,\ldots,i_d)$ and $p=(p_1,\ldots,p_d)$ are a multi-indices. For notational simplicity, we define 
	$\iSet:= \{(i_1,\ldots,i_d)\,|\,i_\iota \in \{1,\ldots,N_\iota\}\} $ as the set of multi-indices.

The computational domain  $\Omega$, also called \emph{physical domain}, is parametrized by the B-Spline basis functions. 
It is given as image of the parameter domain $\p{\Omega}$ under the so-called \emph{geometrical mapping} $G :\; \p{\Omega} \rightarrow \mathbb{R}^{{d}}$, defined as
\begin{align*}
	    G(\xi) := \sum_{i\in \mathcal{I}} P_i \p{N}_{i,p}(\xi),
\end{align*}
with the control points $P_i \in \mathbb{R}^{{d}}$, $i\in \mathcal{I}$.
In order to represent more complicated geometries $\Omega$, multiple non-overlapping domains (patches) $\Omega_n:=G_n(\p{\Omega}),n=1,\ldots,N$ are composed, where each patch is associated with a different geometrical mapping $G_n$. In the following, we refer to such domains $\overline{\Omega}:=\bigcup_{n=1}^N\overline{\Omega}_n$ as \emph{multipatch domains}.

In the IgA concept, the B-Splines  are not only used for representing the geometry, but also as basis for 
finite-dimensional space used for approximating the solution of the PDE.
This motivates to define the basis functions
$\g{N}_{i,p}:=\p{N}_{i,p}~\circ~G^{-1}$
in the physical space 
by mapping the corresponding basis functions $\p{N}_{i,p}$
defined in the parameter domain $\p{\Omega}$.

On each patch $\Omega_{n}$, we now define the local 
IgA
space 
\begin{align}
\label{equ:gVh}
  V_h^{n}:=\text{span}\{\g{N}_{i,p}\}_{i\in\iSet}.
\end{align}
The construction of global 
IgA
space $V_h$ depends on the used formulation, and is given in the next section.

\subsection{Space-time variational formulation and its IgA discretization}

Let $\Omega$ be a bounded Lipschitz domain in $\mathbb{R}^d$,
$d=1,2,$ or $3$, with the boundary $\Gamma = \partial \Omega$.
For any  multi-index $\boldsymbol{\alpha}=(\alpha_1,\ldots,\alpha_d)$ of non-negative integers $\alpha_1,\ldots,\alpha_d$, 
we define the differential operator
$\partial^{\boldsymbol{\alpha}}_{x}=\partial_{x_1}^{\alpha_1}\ldots \partial_{x_d}^{\alpha_d}$,
with $\partial_{x_j} = \partial / \partial x_j$, $j=1,\ldots,d$.
As usual, 
$L_2(\Omega)$ denotes the   Lebesgue  space 
of all Lebesgue measurable and square-integrable functions
endowed with the norm 
$\textstyle{\|v\|_{L_2(\Omega)}} = 
\textstyle{\Big(\int_{\Omega}|v(x)|^2\,dx\Big)^{0.5} }$,
and  $L_{\infty}(\Omega)$ denotes the space of functions that are essentially bounded. 
For a  non-negative integer ${\ell}$,
we define the standard Sobolev space
\begin{equation*}
 	H^{\ell}(\Omega)=\{v \in L_2(\Omega): \partial^{\boldsymbol{\alpha}}_{x} v \in L_{2}(\Omega)\,\text{for all}\,|\boldsymbol{\alpha}| = \textstyle \sum_{j=1}^d\alpha_j \leq \ell \},
\end{equation*}
endowed with the norm
\begin{equation*}
 \| v \|_{H^{\ell}(\Omega)} = \big
  (\sum_{0\leq |\boldsymbol{\alpha}| \leq \ell} \|\partial^{\boldsymbol{\alpha}}_{x}v \|_{L_2(\Omega)}^2
  \big
  )^{\frac{1}{2}},
\end{equation*}
whereas the trace space of $H^{1}(\Omega)$ is denoted by 
$H^{\frac{1}{2}}(\Gamma)$.
Further, we introduce the subspace 
$H^{1}_0(\Omega) =\{v \in H^{1}(\Omega):v=0\,\text{on}\, \Gamma \}$ 
of all functions $v$ from $H^{1}(\Omega)$ with zero traces on $\Gamma$.
We define the spatial gradient by
$\nabla_x v =  (\partial_{x_1} v,\ldots,\partial_{x_d} v)$.
Let $\ell$ and $m$ be positive integers.  
For functions defined in the space-time cylinder $Q$, we define the Sobolev spaces 
\begin{equation*}
H^{\ell,m}(Q)=\{v \in L_2(Q): \partial^{\boldsymbol{\alpha}}_x v \in L_2(Q)\,\text{for}\, 0\leq |\boldsymbol{\alpha}|\leq \ell,
\; \text{and}\;\partial_t^{i} v \in L_2(Q),\,i=1,\ldots,m\},
\end{equation*}
where $\partial_t = \partial / \partial t$,
and, in particular, the subspaces
\begin{align*}
H^{1,0}_0(Q)=&\{v \in L_2(Q):\nabla_x v \in  [L_2(Q)]^d,\, v=0\,\text{on}\, \Sigma\} \;\text{and}\\
H^{1,1}_{0,\bar{0}}(Q)=&\{v \in L_2(Q):\nabla_x v \in  [L_2(Q)]^d,\, \partial_t v \in L_2(Q), 
v=0\,\text{on}\, \Sigma, \,
v=0\,\text{on}\, \Sigma_T\}.
\end{align*}
We equip the above spaces with the norms and seminorms
\begin{equation*}
 \|v\|_{H^{\ell,m}(Q)}=\big(\sum_{|\boldsymbol{\alpha}|\leq \ell}
    \|\partial_{x}^{(\alpha_1,\ldots,\alpha_d)} v \|^2_{L_2(Q)}+
     { \sum_{m_0=0}^{m}  \|\partial_{t}^{m_0} v \|^2_{L_2(Q)}\big)^{\frac{1}{2}}  }
\end{equation*}
and  
\begin{equation*}
  |v|_{H^{\ell,m}(Q)}=\big(\sum_{|\boldsymbol{\alpha}| =\ell}
    \|\partial_{x}^{(\alpha_1,\ldots,\alpha_d)} v\|^2_{L_2(Q)}+
     {  \|\partial_{t}^{m} v\|^2_{L_2(Q)}\big)^{\frac{1}{2}}  },
\end{equation*}
respectively.  

Using the standard procedure and integration by parts with respect to both $x$ and $t$,
we can easily derive the following space-time variational formulation 
of (\ref{HLN:Eqn:ParabolicIBVP-CF}):
find  $u\in H^{1,0}_0(Q)$ such that 
\begin{equation}
\label{HLN:0_sspc_tme_weak}
a(u,v) = l(v) \quad \text{for all}\,v\in H^{1,1}_{0,\bar{0}}(Q),\\
\end{equation}
with the bilinear form 
\begin{equation*}
 a(u,v) = -\int_{Q}u(x,t)\partial_tv(x,t)\,dx\,dt+\int_{Q}\nabla_xu(x,t)\,\cdot\,\nabla_x v(x,t)\,dx\,dt
\end{equation*}
and the linear form
\begin{equation*}
 l(v)  = \int_{Q}f(x,t)v(x,t)\,dx\,dt+\int_{\Omega}u_0(x)v(x,0)\,dx,
\end{equation*}
where the source $f\in L_2(Q)$ and the initial conditions $u_0\in L_2(\Omega)$ are given.

Without loss of generality, we only consider homogeneous 
Dirichlet boundary conditions on $\Sigma$.
The method presented in this paper can easily be generalized to 
other constellations of boundary conditions. 
The space-time variational formulation (\ref{HLN:0_sspc_tme_weak})
has a unique solution, see, e.g, 
\cite{HLN:Ladyzhenskaya:1973a} and \cite{HLN:LadyzhenskayaSolonnikovUralceva:1967a}.
\begin{assumption}
\label{HLN:Assumption1}
 We assume that the  solution $u$ of (\ref{HLN:0_sspc_tme_weak}) belongs to
 $V= H^{1,0}_0(Q)\cap H^{\ell,m}(Q)$ with some 
$\ell \geq  2$ and $m\geq 1$.
\end{assumption}

We describe the space-time 
cylinder $Q$  as a union of  non-overlapping 
time slabs  $Q_1$, $Q_2$,\ldots,$Q_N$. 
We consider a partition $0=t_0<  t_1 < \ldots< t_N=T$ of the time interval $[0,T]$,
and denote the sub intervals by  $J_n=(t_{n-1},t_n)$. 
We now define the time slabs
$Q_n=\Omega\times J_n$ and the faces 
$\Sigma_n=\overline{Q}_{n+1}\cap \overline{Q}_n=\Omega \times \{t_n\}$ 
between the time slabs,
where we identify $\Sigma_T$ and $\Sigma_N$. 
In that way, we have the decomposition
$\overline{Q}=\cup_{n=1}^N \overline{Q}_n$, where each space-time cylinder $Q_n$ has a geometrical mapping $G_n$. To keep the notation simple,  in what follows, we will use 
the sup-index $n$ to  denote the restrictions  to $Q_n$, e.g., $u^n:=u|_{Q_n}$.

\begin{remark}
We note that the spatial domain $\Omega$ can also be a multipatch domain. This leads to a representation of $Q_n$ as union of non-overlapping 
space-time patches
$Q_{n,k}, k=1,\ldots,K$, i.e.,  $\overline{Q}_n=\cup_{k=1}^K \overline{Q}_{n,k}$.  The corresponding bases are then coupled in a conforming way. 
\end{remark}

We denote the    global discontinuous B-Spline space 
and the local continuous patch-wise B-Spline  spaces by 
\begin{equation}
\label{HLN:global_Bspl_spc_time}
 V_{0h}= \{v_h\in L_2(Q):
	v_h|_{Q_n}\in V_h^n,\,\text{for}\, n=1,\ldots,N,\, \text{and}\,v_h|_{ \Sigma} =0\}
\end{equation}
and
\begin{equation}
\label{HLN:local_Bspl_spc_time}
 V_{0h}^{n}= \{v_h\in  V_h^n, \,\text{for}\, n=1,\ldots,N,
	\, \text{and}\,v_h|_{ \Sigma} =0\},
\end{equation}
respectively. Notice that $v_h\in V_{0h}$ is discontinuous across $\Sigma_n$. 
We introduce the notations
\begin{equation*}
	v_{h,+}^n= \lim_{\varepsilon \to 0^+}v_h(t_n+\varepsilon),\;
	v_{h,-}^n= \lim_{\varepsilon \to 0^-}v_h(t_n+\varepsilon),\;
	\llbracket v_h \rrbracket^n = v_{h,+}^n - v_{h,-}^n,\;
	\llbracket v_h \rrbracket^0 = v_{h,+}^0,
\end{equation*}
where $\llbracket v_h \rrbracket^n$ denotes the jump of $v_h$  across $\Sigma_n$ for $n \geq 1$, and 
$\llbracket v_h \rrbracket^0 = v_{h,+}^0$ denotes 
the trace of $v_h$ on $\Sigma_0$.
For a smooth function $u$, we obviously have 
$	\llbracket u \rrbracket^{n} =u^{n}_{+}-u^n_{-}=0\,\text{ for}\, n\geq 1,$ 
and 
$
	\llbracket u \rrbracket^0 =u|_{\Sigma_0}.
$

Let us now consider the space-time slab $Q_n$, 
and let us denote the outer normal  to 
 $\partial Q_n$
by 
$\mathbf{n}=(n_1,\ldots,n_d,n_{d+1})=(\mathbf{n}_x,n_t)$.  
For the time being, we assume 
that $u^{n-1}$ is known. Let $v^n_h\in V_{0h}^{n}$ and  $w^{n}_h=v^{n}_h+\theta_n\,h_{n}\partial_t v^{n}_h$
with some positive parameter $\theta_n$, 
which will be defined later. 
We note that $w^{n}_h\big|_{\Sigma}=0$. 
Multiplying 
$\partial_t u - \Delta u = f$
by $w^{n}_h$, integrating over $Q_n$,  and  applying integration by parts,
we arrive at the variational identity
\begin{align*}
& \int_{Q_n} (\partial_t\,u\,(v_h^{n}+\theta_n\,h_{n}\partial_t v_h^{n}) + \nabla_x\,u\cdot\nabla_x\,v_h^{n}+
	\theta_n\,h_{n}\nabla_x u\cdot\nabla_x\partial_tv_h^{n})\,dx\,dt \\
	\nonumber
&	- \int_{\partial Q_n }\,n_x\cdot \nabla_x u(v_h^{n}+\theta_n\,h_{n}\partial_t v_h^{n})\,dx
+ \int_{\Sigma_{n-1}} u^{n-1}_{+}\,v_{h,+}^{n-1}\,dx\\
\nonumber
& = \int_{Q_n}f\,(v_h^{n}+\theta_n\,h_{n}\partial_t v_h^{n})\,dx\,dt+\int_{\Sigma_{n-1}} u^{n-1}_{-}\,v_{h,+}^{n-1}\,dx
\end{align*}
for $n = 1,\ldots,N$,
where we used that  $u^{n-1}_{-}=u^{n-1}_{+}=u^{n-1}$ on every 
$\Sigma_{n-1}$.
Furthermore,  
using  $n_x|_{\Sigma_n}=0$
and  $w_h=0$ on $\Sigma$, we have 
\begin{align*}
a_{Q_n}(u,v_h):=&\int_{Q_n}(\partial_t\,u\,(v_h^{n}+\theta_n\,h_{n}\partial_t v_h^{n}) + \nabla_x\,u\cdot\nabla_x\,v_h^{n}+
	\theta_n\,h_{n}\nabla_x u\cdot\nabla_x\partial_tv_h^{n})\,dx\,dt \\
	\nonumber
 & +\int_{\Sigma_{n-1}}\llbracket u\rrbracket^{n-1} \,v_{h,+}^{n-1}\,dx=
\int_{Q_n}f\,(v_h^{n}+\theta_n\,h_{n}\partial_t v_h^{n})\,dx\,dt,
\end{align*}
for all $n=2,\ldots, N$, and 
\begin{align*}
a_{Q_1}(u,v_h):=&\int_{Q_1} (\partial_t\,u\,(v_h^{1}+\theta_1\,h_{1}\partial_t v_h^{1}) + \nabla_x\,u\cdot\nabla_x\,v_h^{1}+
	\theta_1\,h_{1}\nabla_x u\cdot\nabla_x\partial_tv_h^{1})\,dx\,dt \\
	\nonumber
 & +\int_{\Sigma_{0}}\llbracket u\rrbracket^{0} \,v_{h,+}^{0}\,dx=
\int_{Q_1}f\,(v_h^{1}+\theta_1\,h_{1}\partial_t v_h^{1})\,dx\,dt + 
\int_{\Sigma_0}u_0\,v_{h,+}^{0}\,dx.
\end{align*}
Summing over all $Q_n$, we conclude that
\begin{equation}
\label{HLN:DGIGA_Schme_03}
a_h(u,v_h)= l_h(v_h),\quad \forall v_h\in V_{0h},
\end{equation}
where 
\begin{equation*}
a_h(u,v_h) = \sum_{n=1}^Na_{Q_n}(u,v_h) 
\end{equation*}
and 
\begin{equation*}
l_h(v_h) =\sum_{n=1}^N \int_{Q_n}f\,(v^n_h+\theta_n\,h_n\partial_t v^n_h)\,dx\,dt +
\int_{\Sigma_0}u_0\,v_{h,+}^{0}\,dx.
\end{equation*}
Now, the space-time dG IgA variational scheme for 
(\ref{HLN:Eqn:ParabolicIBVP-CF})
reads
as follows: 
Find $u_h\in V_{0h}$ such that
\begin{equation}
\label{HLN:DGIGA_Schme}
 a_h(u_h,v_h) = l_h(v_h),\quad \forall v_h\in V_{0h}.
\end{equation}

Motivated by 
the definition of the bilinear form $a_h(\cdot,\cdot)$ in (\ref{HLN:DGIGA_Schme}),
we introduce the mesh-dependent dG norm
\begin{align*}
	\|v\|_{dG}:=\Big(\sum_{n=1}^N\Big(\|\nabla_x v\|^2_{L_2(Q_n)}+\theta_n\,h_{n}\,\|\partial_tv\|^2_{L_2(Q_n)}+
           \frac{1}{2}\|\llbracket v \rrbracket^{n-1} \|^2_{L_2(\Sigma_{n-1})}\Big)\ +\frac{1}{2}\|v\|^2_{L_2(\Sigma_{N})} \Big)^{\frac{1}{2}},
\end{align*}
In the following, we recall some important properties of the 
IgA scheme (\ref{HLN:DGIGA_Schme}) respectively 
the bilinear form $a_h(\cdot,\cdot)$.
For the proofs, we refer to \cite{HLN:HoferLangerNeumuellerToulopoulos:2017a}.
\begin{lemma}
	The 
	bilinear form $a_h(\cdot,\cdot)$, 
	defined in (\ref{HLN:DGIGA_Schme}), 
	is
	$V_{0h}$-elliptic, i.e., 
	\begin{align}\label{HLN:Coercivity}
		a_h(v_h,v_h)\geq C_e\|v_h\|_{dG}^2, \quad\text{for}{\ }v_h\in V_{0h},
	\end{align}
where 
$C_e= 0.5$
for $\theta_n \leq  C_{inv,0}^{-2}$, 
with the positive, $h_n$-independent constant $C_{inv,0}$ from the inverse inequality 
\begin{equation*}
 \|v_h\|^2_{L_2(\Sigma_{n-1})} \le C_{inv,0} h_n^{-1} \|v_h\|^2_{L_2(Q_{n})}
\end{equation*}
that holds for all $v_h \in V_h^n$, $n=1,\ldots,N$.
\end{lemma}
The 
$V_{0h}$-ellipticity of the bilinear form $a_h(\cdot,\cdot)$
implies that there exists a unique solution to \eqref{HLN:DGIGA_Schme_03}. In order to obtain a priori error estimates, we introduce the space $V_{0h,*}=V+V_{0h}$ endowed with the norm
\begin{align}\label{HLN:Norm_Voh*}
	\|v\|_{dG,*}:=\Big(\|v\|^2_{dG}+\sum_{n=1}^N(\theta_n h_{n})^{-1}\|v\|^2_{L_2(Q_n)}
     + \sum_{n=2}^N\|v_{-}^{n-1}\|^2_{L_2(\Sigma_{n-1})}	\Big)^{\frac{1}{2}}.
\end{align}
\begin{lemma}\label{HLN:lemmaA_h_continuity}
 Let $u\in V_{0h,*}$. Then the boundedness inequality
 \begin{align}\label{A_h_con_1}
  |a_h(u,v_h)| \leq C_b \|u\|_{dG,*}\|v_h\|_{dG}
 \end{align}
 holds for all $v_h\in V_{0h}$,
 where $C_b=\max(C_{inv,1}\,\theta_{max},2)$, 
 with  $\textstyle{\theta_{max}=\max_{n}\{\theta_n\}\leq C^{-2}_{inv,0}}$ 
 and the positive, $h_n$-independent constant $C_{inv,1}$ from the inverse inequality
\begin{equation*}
 \|\partial_t \partial_{x_i} v_h\|^2_{L_2(Q_{n})} \le C_{inv,1} h_n^{-2} \|\partial_{x_i} v_h\|^2_{L_2(Q_{n})}
\end{equation*}
that holds for all $v_h \in V_h^n$, $n=1,\ldots,N$, $i=1,\ldots,n$.
\end{lemma}

\begin{theorem}\label{HLN:theorem_1ST}
Let $u$  and $u_h$ solve (\ref{HLN:0_sspc_tme_weak})
and (\ref{HLN:DGIGA_Schme}), respectively.
Under the regularity Assumption \ref{HLN:Assumption1}, 
	there exists a positive generic constant $C$, 
	which is independent of $h = \max\{h_n\}$, 
	such that 
	\begin{align}\label{HLN:Theorem1_1ST}
		\|u-u_h\|_{dG}\leq  C( h^{\ell-1}+ h^{m-\frac{1}{2}})\,\|u\|_{H^{\ell,m}({Q})}.
	\end{align}
Moreover,  if $1\leq m<\ell\leq p+1$, then 
\begin{align}\label{HLN:theorem1_2ST}
		\|u-u_h\|_{dG}\leq C h^{m-\frac{1}{2}}\|u\|_{H^{\ell,m}({Q})}.
	\end{align}
\end{theorem}
\begin{remark}\label{HLN:Estim_MaxRegul}
We remark that, for the case of highly smooth solutions, i.e., $p+1 \leq \min (\ell,m)$,  
estimate (\ref{HLN:Theorem1_1ST}) takes the form
\begin{align}\label{HLN:Theorem1_5ST}
 \|u-u_h\|_{dG}\leq  C\, h^{p}\,\|u\|_{H^{\ell,m}({Q})}.
	\end{align}	
\end{remark}

\subsection{Efficient Matrix Assembly}

Let us recall the 
IgA
variational problem given in \eqref{HLN:DGIGA_Schme}. 
The local bilinear form for each space-time slab $Q_n$
is given by
\begin{align*}
 a_{Q_n}(u_h,v_h) =& \int_{Q_n}\partial_t\,u_h^n\,(v_h^{n}+\theta_n\,h_{n}\partial_t v_h^{n}) + \nabla_x\,u_h^n\cdot\nabla_x (v_h^{n}+\theta_n\,h_{n}\partial_t v_h^{n})\,dx\,dt\\ & +\int_{\Sigma_{n-1}} u^{n-1}_{h,+}\,v_{h,+}^{n-1}\,ds 
 - \int_{\Sigma_{n-1}} u^{n-1}_{h,-}\,v_{h,+}^{n-1}\,ds \\
 =:& b_{Q_n}(u^n_h,v^n_h) - \int_{\Sigma_{n-1}} u^{n-1}_{h,-}\,v_{h,+}^{n-1}\,ds,
\end{align*}
where $n = 1,\ldots,N$.
For the local spaces $V^{n}_{0h}$
defined by (\ref{HLN:local_Bspl_spc_time}), 
we now introduce the simpler notation $\varphi_j^n$ for the B-Spline basis functions 
such that
\begin{align*}
 V^{n}_{0h} = \text{span}\{ \varphi_j^n \}_{j = 1}^{N_n}
\end{align*}
for $n = 1,\ldots,N$.
Once the basis is chosen, 
from the IgA variational scheme \eqref{HLN:DGIGA_Schme},
we immediately obtain 
the linear system
\begin{align}\label{HLN:equ_linear_system}
 \mathbf{L}_h \vec{u}_h := 
 \begin{pmatrix}
    \mathbf{A}_1 \\
    -\mathbf{B}_2 & \mathbf{A}_2 \\
    && \ddots & \ddots \\
    &&& -\mathbf{B}_N & \mathbf{A}_N
 \end{pmatrix}
 \begin{pmatrix}
  \vec{u}_1 \\
  \vec{u}_2 \\
  \vdots \\
  \vec{u}_N
 \end{pmatrix}
= 
 \begin{pmatrix}
  \vec{f}_1 \\
  \vec{f}_2 \\
  \vdots \\
  \vec{f}_N
 \end{pmatrix} =: \vec{f}_h,
\end{align}
 with the matrices 
\begin{align*}
 \mathbf{A}_n[i,j] := b_{Q_n}(\varphi^n_j,\varphi^n_i) \quad  \text{for } i,j = 1,\ldots,N_n
\end{align*}
on the diagonal for $n=1,\ldots,N$,
and the matrices 
\begin{align*}
  \mathbf{B}_n[i,k] := \int_{\Sigma_{n-1}} \varphi^{n-1}_{k,-}\,\varphi_{i,+}^{n-1}\,ds \quad \text{for } k = 1,\ldots,N_{n-1} \text{ and } i = 1,\ldots,N_n.
\end{align*}
on the lower off diagonal for $n=2,\ldots,N$.
Moreover, the right hand sides are given 
by
\begin{align*}
 \vec{f}_n[i] := l_h(\varphi^n_i), \quad 
 i = 1,\ldots,N_n,
\end{align*}
for $n=1,\ldots,N$.

If the geometrical mappings $G_n: \widehat{Q}\rightarrow Q_n, n=1,\ldots,N$, preserve the tensor product structure of the IgA basis functions $\varphi_i^n$, we can use this information to save assembling time and storage costs for the linear system \eqref{HLN:equ_linear_system}. 
In this case, we can write the basis functions $\varphi_i^n$ in the form
\begin{align*}
 \varphi_i^n(x,t) = \phi_{i_x}^n(x) \psi_{i_t}^n(t) \quad \text{with } i_x \in \{ 1,\ldots,N_{n,x} \} \text{ and } i_t \in \{ 1,\ldots,N_{n,t} \},
\end{align*}
where $N_n = N_{n,x} N_{n,t}$. Using this representation, we can write the matrices $\mathbf{A}_n, n = 1,\ldots,N$ as
\begin{align}
\label{equ:structureAn}
 \mathbf{A}_n = {\mathbf{K}}_{n,t}\otimes {\mathbf{M}}_{n,x}   +   {\mathbf{M}}_{n,t} \otimes {\mathbf{K}}_{n,x},
\end{align}
with the standard mass and stiffness matrices with respect to space
\begin{align*}
   {\mathbf{M}}_{n,x}[i_x,j_x] &:= \int_{\Omega} \phi_{j_x}^n \phi_{i_x}^n \, dx, \qquad {\mathbf{K}}_{n,x}[i_x,j_x] := \int_{\Omega} \nabla_x\phi_{j_x}^n \cdot \nabla_x\phi_{i_x}^n \,dx,
\end{align*}
where $i_x,j_x = 1,\ldots,N_{n,x}$, and corresponding matrices with respect to time
\begin{align}
\label{equ:TimeMatrices}
\begin{split}
  {\mathbf{K}}_{n,t}[i_t, j_t] &:= \int_{t_{n-1}}^{t_n} \partial_{t}\psi_{j_t}^n (\psi_{i_t}^n + \theta_n h_n \partial_{t}\psi_{i_t}^n) \,dt + \psi_{j_t}^n(t_{n-1}) \psi_{i_t}^n(t_{n-1}), \\
 {\mathbf{M}}_{n,t}[i_t, j_t] &:= \int_{t_{n-1}}^{t_n} \psi_{j_t}^n (\psi_{i_t}^n + \theta_n h_n \partial_{t}\psi_{i_t}^n) \,dt,
\end{split}
 \end{align}
with $i_t,j_t = 1,\ldots,N_{n,t}$.
The matrices on the off diagonal $\mathbf{B}_n, n=2,\ldots,N$, can be written in the form
\begin{align*}
 \mathbf{B}_n :=   {\mathbf{N}}_{n,t} \otimes {\widetilde{\mathbf{M}}}_{n,x},
\end{align*}
with the matrices
\begin{align*}
 {\widetilde{\mathbf{M}}}_{n,x}[i_x,k_x] &:= \int_{\Omega}\phi_{k_x}^{n-1} \phi_{i_x}^n\,dx \quad \text{and} \quad {\mathbf{N}}_{n,t}[i_t, k_t] := \psi_{k_t}^{n-1}(t_{n-1}) \psi_{i_t}^n(t_{n-1}),
\end{align*}
where $i_x = 1,\ldots,N_{n,x}$, $k_x = 1,\ldots,N_{n-1,x}$, $i_t = 1,\ldots,N_{n,t}$ and $k_t = 1,\ldots,N_{n-1,t}$.

\section{Solvers for space-time problems}
\label{sec:solvers}
This section aims at the development of an efficient solver for 
the huge space-time system \eqref{HLN:equ_linear_system}.
Our new solver is based on the time parallel multigrid method 
proposed
in \cite{HLN:GanderNeumueller:2016a}, see also the PhD~thesis \cite{HLN:Neumueller:2013a}. 
The key point in realizing the method efficiently is the application of the smoother, which is the most costly part of the algorithm. The goal is to utilize the structure of the involved matrix $\mathbf A_n^{-1}$, which then allows for a faster application.

\subsection{Time-parallel multigrid}

We want to give an overview of the 
time-parallel 
multigrid method introduced in \cite{HLN:Neumueller:2013a}. Multigrid consists of three main ingredients: the coarse grid solver, the smoother and the prolongation/restriction operators. Concerning the restriction and prolongation operator, it is advantageous to consider coarsening in space and in time separately. 
The  restriction 
in time direction is realized by combining two consecutive time-slabs into a single one. For a more detailed discussion on how space and time coarsening can be combined, we refer to \cite{HLN:Neumueller:2013a}.

In this work, we are mostly interested in the smoother, which is of (inexact) damped block Jacobi type, i.e.,
\begin{align*}
	 u_h^{k+1} = u_h^{k} + \omega \mathbf{D}_h^{-1} \left[ f_h - \mathbf{L}_h u_h^k \right] \quad \text{for } k=1,2,\ldots.
\end{align*}
We use the block diagonal matrix $\mathbf{D}_h := \text{diag}\{ \mathbf{A}_n \}_{n=1}^N$ and the damping parameter $\omega = \frac{1}{2}$, see also \cite{HLN:GanderNeumueller:2016a}. The application of the smoother can be accelerated by replacing the inverse of $\mathbf{D}_h$ by some approximation, i.e., an approximation $\hat{\mathbf{A}}_n^{-1}$ to  $\mathbf{A}_n^{-1}$. The aim of this work is to find a procedure, which allows an efficient application 
of $\hat{\mathbf{A}}_n^{-1}$ to a vector.
In order to achieve this, we will heavily exploiting the 
special tensor structure of $\mathbf A_n$.

\subsection{General construction of an approximation for $\mathbf A_n^{-1}$}
\label{sec:constructionAinv}

In this section, for notational simplicity, we drop the subscript $n$ when considering matrices and vectors defined on the space-time slice $Q_n$. 
We recall the structure of the matrix $\mathbf A$, 
\begin{align*}
	\mathbf A = \mathbf K_t \otimes \mathbf M_x + \mathbf M_t \otimes \mathbf K_x,
\end{align*}
where the matrices $\mathbf M_x$ and $\mathbf K_x$ are symmetric and positive definite, 
while the matrices $\mathbf K_t$ and  $\mathbf M_t$ are non-symmetric, cf. \eqref{equ:structureAn}. 
The matrices $\mathbf M_x$ and $\mathbf K_x$ correspond to $d$-dimensional problem, 
whereas $\mathbf K_t$ and  $\mathbf M_t$  
are only 
related 
to a one dimensional problem 
in one time-slice. 
Hence, the size of the latter two matrices is much smaller than the first two. 
The idea is to use already available preconditioners for symmetric positive definite problems of the form $\mathbf K_x+ \gamma \mathbf M_x$ with $\gamma >0$ to construct efficient and robust preconditioners for $\mathbf A^{-1}$.
The ideas of this section are based on the results developed in \cite{HLN:Tani:2017a} and  \cite{HLN:SangalliTani:2016a}.

We will achieve this by performing a decomposition of ${\mathbf{M}}_{t}^{-1}{\mathbf{K}}_{t}$ using one of the three following methods: \emph{Diagonalization}, \emph{Complex-Schur decomposition}, \emph{Real-Schur decomposition}. 
We obtain a decomposition of the form 
${\mathbf{M}}_{t}^{-1}{\mathbf{K}}_{t}= \mathbf X^{-1} \mathbf Z \mathbf X$, 
where the entries of the matrices $\mathbf X$ and $\mathbf Z$ are complex or real numbers, and $\mathbf Z$ has some sort of ``simple'' structure. A detailed specification will be presented in \secref{sec:diagonalization}, \secref{sec:complexSchur} and \secref{sec:realSchur}.

By defining $\mathbf Y:= (\mathbf M_t\mathbf X)^{-1}$, we obtain the following representations
 \begin{align*}
 	\mathbf M_t = \mathbf Y^{-1}\mathbf X^{-1} \quad \text{and}\quad \mathbf K_t= \mathbf Y^{-1}\mathbf Z \mathbf X^{-1}.
 \end{align*}
Now we can rewrite $\mathbf A$ in the  form
\begin{align*}
	 \mathbf{A} &= {\mathbf{K}}_{t} \otimes {\mathbf{M}}_{x} +{\mathbf{M}}_{t} \otimes {\mathbf{K}}_{x}\\
	 &=(\mathbf Y^{-1}\mathbf Z \mathbf X^{-1}) \otimes {\mathbf{M}}_{x} +(\mathbf Y^{-1}\mathbf X^{-1}) \otimes {\mathbf{K}}_{x}\\
	 &=  (\mathbf Y^{-1}\otimes\mathbf I)\cdot(\mathbf Z\otimes \mathbf M_x + \mathbf I\otimes \mathbf K_x)\cdot(\mathbf X^{-1}\otimes \mathbf I).
\end{align*}
Using the well-known fact that 
$(\mathbf Y^{-1}\otimes\mathbf I)^{-1} = \mathbf Y \otimes\mathbf I$
and 
$(\mathbf X^{-1}\otimes\mathbf I)^{-1} = \mathbf X \otimes\mathbf I$,
we obtain
\begin{align}
   \label{equ:ST:generalDecompA}
	 \mathbf{A}^{-1} = (\mathbf X\otimes \mathbf I)\cdot(\mathbf Z\otimes \mathbf M_x + \mathbf I\otimes \mathbf K_x)^{-1}\cdot(\mathbf Y\otimes\mathbf I).
\end{align}
In the subsequent subsections, we will investigate the structure of the matrix $(\mathbf Z\otimes \mathbf M_x + \mathbf I\otimes \mathbf K_x)$ for each of the decomposition methods, and we will look for efficient ways of (approximate) inversion.

In the following, the generalized eigenvalues $\lambda_i := \alpha_i+\imath\beta_i \in\mathbb{C}$ of $(\mathbf K_t,\mathbf M_t)$, i.e.,
\begin{align}
\label{equ:ST:generalizedEigenvalue}
	\mathbf K_t \mathbf z_i = \lambda_i \mathbf M_t\mathbf z_i,
\end{align}
with the eigenvector
$\mathbf z := \mathbf x + \imath \mathbf y$,
will play an important role for constructing an efficient application of \eqref{equ:ST:generalDecompA}. 
First of all, for $0 < \theta_n\leq C_{inv}^{-2}$, 
where $C_{inv}$ denotes the constant from the inverse inequality
\begin{align}
\label{equ:trace1D}
	|v(t_{n-1})|^2\leq C_{inv}^2h_n^{-1}\|v\|^2_{L_2(t_{n-1},t_n)} 
	\;\forall v \leftrightarrow \mathbf v\in\mathbb{R}^{N_t},
\end{align}
we have the positiveness of the matrices $\mathbf K_t$ and $\mathbf M_t$, 
see \cite{HLN:WarburtonHesthaven:2003} for an explicit formula of $C_{inv}= C_{inv}(p)$ in the case of polynomials of the degree $p$.
\begin{lemma}
\label{lem:pos_KM}
Let $\mathbf K_t$ and $\mathbf M_t$ 
 be given by
\eqref{equ:TimeMatrices}, 
and let the constant $C_{inv} >0$ be defined according to \eqref{equ:trace1D}. 
If $\theta_n > 0$, then the matrix $\mathbf K_t$ is positive, i.e., 
$\mathbf v^T\mathbf K_t\mathbf v>0$ for all 
$\mathbf v\in\mathbb{R}^{N_t} \setminus \{\mathbf 0\}$, 
and if $\theta_n <  2C_{inv}^{-2}$, then the matrix $\mathbf M_t$ is positive.
\end{lemma}
\begin{proof}
We first consider the matrix $\mathbf K_t$. 
We can write $\mathbf v^T \mathbf K_t \mathbf v$ in the following way:
 	\begin{align*}
	  \mathbf v^T\mathbf K_t\mathbf v &= (\mathbf K_t \mathbf v, \mathbf v) 
	  = \int_{t_{n-1}}^{t_n} (v'(t)v(t)+\theta_nh_n(v'(t))^2)\,dt + |v(t_{n-1})|^2 \\
					  &= \theta_n h_n\|v'\|_{L_2(t_{n-1},t_n)}^2 + \frac{1}{2}\int_{t_{n-1}}^{t_n} (v^2)'(t)\,dt + |v(t_{n-1})|^2 \\
					  &= \theta_n h_n\|v'\|_{L_2(t_{n-1},t_n)}^2 + \frac{1}{2}|v(t_{n})|^2 -\frac{1}{2}|v(t_{n-1})|^2 +|v(t_{n-1})|^2\\
					  &= \theta_n h_n\|v'\|_{L_2(t_{n-1},t_n)}^2 + \frac{1}{2}(|v(t_{n})|^2+|v(t_{n-1})|^2) >0.
 	\end{align*}
for all $v \leftrightarrow \mathbf v\in \mathbb{R}^{N_t} \setminus \{\mathbf 0\}$.
Using \eqref{equ:trace1D}, we similarly obtain
\begin{align*}
	  \mathbf v^T\mathbf M_t\mathbf v &= (\mathbf M_t \mathbf v, \mathbf v) 
	  = \int_{t_{n-1}}^{t_n} (v(t)^2+\theta_nh_n v'(t) v(t))\,dt\\
					  &= \|v\|_{L_2(t_{n-1},t_n)}^2 + \frac{1}{2}\theta_n h_n(|v(t_{n})|^2 -|v(t_{n-1})|^2) \\
					  &\geq \left(1-\frac{C_{inv}^{2}\theta_n}{2}\right)\|v\|_{L_2(t_{n-1},t_n)}^2 +\frac{1}{2}\theta_n h_n|v(t_{n})|^2 >0.
\end{align*}
for all $v \leftrightarrow \mathbf v\in \mathbb{R}^{N_t} \setminus \{\mathbf 0\}$. 
$\hfill \square$
 \end{proof}
Next we 
are going 
to investigate the generalized eigenvalues in \eqref{equ:ST:generalizedEigenvalue}. 
More precisely,
we want to find conditions 
under which
the real part $\alpha$ is positive. 
However, for a 
generalized eigenvalue problem
$\mathbf A \mathbf z = \lambda \mathbf B \mathbf z$, this does not follow from the positivity of $\mathbf A$ and $\mathbf B$ 
as following example shows.
\begin{example}
	Let the matrices $\mathbf{A}$ and $\mathbf B$ be given by
	\begin{align*}
		\mathbf{A} = \MatTwo{5}{-2}{13}{18} \quad \text{and} \quad \mathbf B = \MatTwo{4}{10}{-10}{9}.
	\end{align*}
	For the spectra, 
	we have
	$\sigma(\mathbf A) = \{9\pm 2\sqrt5\imath\}$ and $\sigma(\mathbf{B}) = \{\frac{13}{2}\pm5\sqrt{15}\imath\}$. However, the generalized eigenvalues are $\sigma(\mathbf B^{-1} \mathbf A)=\{-\frac{103}{272}\pm\sqrt{4435}\imath\}$.
\end{example}

Let $\mathbf z$ be the eigenvector to the eigenvalue $\lambda=\alpha+\imath \beta$, i.e.,  $(\mathbf A - \lambda\mathbf B)\mathbf z = 0$. 
Multiplying from the left with $(\mathbf x- \imath\mathbf y)^T$ 
yields
\begin{align*}
	(\mathbf x- \imath\mathbf y)^T(\mathbf A - (\alpha+\imath \beta)\mathbf B)\mathbf (\mathbf x+ \imath\mathbf y) = 0.
\end{align*}
Separating the real and imaginary part, we obtain
\begin{align}
\label{equ:genEVLinerSystem}
\begin{split}
	\alpha (\mathbf x^T \mathbf B\mathbf x + \mathbf y^T \mathbf B\mathbf y) - \beta (\mathbf x^T( \mathbf B - \mathbf B^T) \mathbf y) &= \mathbf x^T \mathbf A\mathbf x + \mathbf y^T \mathbf A\mathbf y \\
	\alpha (\mathbf x^T( \mathbf B -\mathbf  B^T) \mathbf y) + \beta (\mathbf x^T \mathbf B\mathbf x + \mathbf y^T \mathbf B\mathbf y) &= \mathbf x^T( \mathbf A - \mathbf A^T) \mathbf y.
\end{split}
\end{align}
Introducing the abbreviations
$a :=  \mathbf x^T \mathbf A\mathbf x + \mathbf y^T \mathbf A\mathbf y$, $b :=\mathbf x^T \mathbf B\mathbf x + \mathbf y^T \mathbf B\mathbf y$, $c := \mathbf x^T( \mathbf B - \mathbf B^T) \mathbf y$ and $d := \mathbf x^T( \mathbf A - \mathbf A^T) \mathbf y$, 
we can rewrite this system in the compact 
form
\begin{align*}
	\MatTwo{b}{-c}{c}{b}\VecTwo{\alpha}{\beta}=\VecTwo{a}{d},
\end{align*}
and $\alpha$ is then given by the formula
\begin{align}
\label{equ:formulaMu_general}
	\alpha = \frac{1}{b^2+c^2}(ab  + cd).
\end{align}
We can easily observe the statements of the following lemma.
\begin{lemma}
\label{lem:EV_Basicproperties}
	Let $\mathbf A$ and $\mathbf B$ be positive matrices, then the following statements hold:
	\begin{enumerate}
		\item $a>0$ and $b>0$
		\item If $\beta=0$, i.e., the eigenvalue $\lambda\in\mathbb{R}$, then $\lambda = \alpha>0$.
		\item If either $\mathbf A$ or $\mathbf B$ are symmetric, then $\alpha>0$.
	\end{enumerate}
	If $\mathbf A$ is only non-negative, then these inequalities hold with 
	$\geq$ instead of $>$.
\end{lemma}
\begin{proof}
	The positivity of $a$ and $b$
	immediately
	follows 
	from the definition. If the eigenvalue $\lambda$ is real, i.e., $\beta=0$, we obtain from the first equation of \eqref{equ:genEVLinerSystem} that 
	$\alpha = a/b>0$. 
	If either $\mathbf A$ or $\mathbf B$ is symmetric, then either $d$ or $c$
	is zero. 
	Hence, by \eqref{equ:formulaMu_general}, $\alpha$ is positive.
 	$\hfill \square$
	\end{proof}

Let us  now consider  the special case of $\mathbf A = \mathbf K_t$ and $\mathbf B = \mathbf M_t$. For notational simplicity, we drop the subscript $n$, and consider the interval $[0,T]$. 
First we observe that 
\begin{align*}
	c &= \mathbf x^T( \mathbf B - \mathbf B^T) \mathbf y = \theta h \int_0^Ty'(t)x(t)-x'(t)y(t)\,dt\\
	d &= \mathbf x^T( \mathbf A - \mathbf A^T) \mathbf y = \int_0^Tx'(t)y(t)-y'(t)x(t)\,dt.
\end{align*}
Hence, it follows that $c = -\theta h d$. 
This relation leads to the following formula for $\alpha$:
\begin{align}
	\label{equ:formulaMu_specific}
	\alpha = \frac{1}{b^2+c^2}(ab-\theta h d^2).
\end{align}
The problem then reduces to check the relation $ab - \theta h d^2>0$, 
which then reads as
\begin{align}
\label{equ:formulaMu_specific2}
	(\mathbf x^T \mathbf A\mathbf x + \mathbf y^T \mathbf A\mathbf y)(\mathbf x^T \mathbf B\mathbf x + \mathbf y^T \mathbf B\mathbf y)-\theta h\mathbf 
	(x^T( \mathbf A - \mathbf A^T) \mathbf y)^2
	> 0,
\end{align}
for the eigenvector $\mathbf z = \mathbf x+\imath \mathbf y$ corresponding to $\lambda = \alpha + \imath \beta$. 
Rewriting \eqref{equ:formulaMu_specific2} in terms of functions, 
we get the relation
\begin{align*}
&\left(\theta_n h_n\|x'\|^2 + \frac{1}{2}(|x(T)|^2+|x(0)|^2)+\theta_n h_n\|y'\|^2 + \frac{1}{2}(|y(T)|^2+|y(0)|^2)\right)\\
&\cdot\left(\|x\|^2 + \frac{1}{2}\theta_n h_n(|x(T)|^2 -|x(0)|^2)+\|y\|^2 + \frac{1}{2}\theta_n h_n(|y(T)|^2 -|y(0)|^2)\right)\\
 & -\theta h\left(\int_0^Tx'(t)y(t)-y'(t)x(t)\,dt\right)^2>0
\end{align*}
Unfortunately, in this work, we cannot give a complete characterization of the conditions under which the last inequality holds.

Let us consider the special case $\theta=0$. First of all, we note that 
	$\mathbf v^T\mathbf K_t\mathbf v= \frac{1}{2}(|v(t_{n-1})|^2 + |v(t_{n})|^2),$
which then only defines a seminorm. Hence, discrete coercivity is not valid. 
Therefore, this case is not covered by the analysis presented in \cite{HLN:HoferLangerNeumuellerToulopoulos:2017a}. For its analysis, we refer to \cite{HLN:Steinbach:2015a}, where an inf-sup condition and error estimates are proven.
The matrix $\mathbf M_t$ is symmetric and $\mathbf v^T\mathbf M_t\mathbf v = \|v\|^2_{L^2}$. From this fact, we can deduce the following statement by means of \lemref{lem:EV_Basicproperties}:
\begin{proposition}
\label{prop:theta0}
	Let $\mathbf K_t$ and $\mathbf M_t$ be as defined above with $\theta = 0$. 
	Then $\alpha\geq0$.
\end{proposition}
\begin{remark}
\label{rem:EV_realPart_1}
	In the condition number analysis of the following subsections, we consider matrices of the form $\mathbf K_x + \alpha \mathbf M_x$, which are required to be positive definite. Therefore, the positivity of $\alpha$ can be relaxed in the case that $|\Gamma_D|>0$.
\end{remark}

\begin{remark}
\label{rem:EV_realPart_2}
	A more detailed investigation of \eqref{equ:genEVLinerSystem} shows that
	\begin{align}
		\alpha = 0 \Longleftrightarrow x(0)=x(T)=y(0)=y(T) = 0,
	\end{align}
	for the eigenvector $\mathbf z = \mathbf x + \imath \mathbf y$ corresponding to $\alpha+\imath \beta$.
	
	For the case $p=1$, one can even show that for an eigenvector corresponding to an purely imaginary eigenvalue the property $x(0)=x(T)=y(0)=y(T) = 0$ cannot hold. Considering a uniform knot vector in $[0,1]$ with B-Splines of degree $p=1$ and $N_t\geq 3$, it holds
	\begin{align*}
		\mathbf K_t = \frac{1}{2}\begin{bmatrix}
		1 & 1\\
-1 & 0 & 1\\
 & \ddots & \ddots & \ddots\\
 &  & -1 & 0 & 1\\
 &  &  & -1 & 0 & 1\\
 &  &  &  & -1 & 1
		                         \end{bmatrix} \quad \text{and}\quad
			\mathbf M_t = C_n\begin{bmatrix}
2 & 1\\
1 & 4 & 1\\
 & \ddots & \ddots & \ddots\\
 &  & 1 & 4 & 1\\
 &  &  & 1 & 4 & 1\\
 &  &  &  & 1 & 2
		                         \end{bmatrix},                         
	\end{align*}
	where $C_n>0$ depends on $N_t$.
Rewriting $\mathbf K_t \mathbf z = \imath\beta \mathbf M_t \mathbf z$ as recurrence relation for $\mathbf z=[z_1,z_2,\ldots,z_{N_t-1},z_{N_t}]$, we obtain
\begin{align}
\label{equ:recurrenceP1}
\begin{split}
	z_1 + z_2 &= \imath\beta (2z_1 + z_2)\\
	-z_{i-1} + z_{i+1} &= \imath\beta (z_{i-1}+ 4z_{i} + z_{i+1}) \quad i=2,\ldots,N_t-1\\
	-z_{N_t-1} + z_{N_t} &= \imath\beta (z_{N_t-1} + 2z_{N_t}),
	\end{split}
\end{align}
where we put the real number $C_n$ and the $1/2$ in front of $\mathbf K_t$ into the eigenvalue $\imath\beta$.  
In order for $\mathbf z=[0,z_2,\ldots,z_{N_t-1},0]$ to be an eigenvector, we obtain from the first
line of \eqref{equ:recurrenceP1}
\begin{align*}
	z_2 = \imath\beta z_2 \Leftrightarrow (1-\imath\beta)z_2=0.
\end{align*}
Since $(1-\imath\beta)$ cannot be zero, the only possibility for this equation to hold is when $z_2 =0$. Considering now the second line of \eqref{equ:recurrenceP1} and assuming $z_1 =\ldots = z_{j} =0$,
then, for $i=j$,
the equation reads
\begin{align*}
	z_{j+1} = \imath\beta z_{j+1} \Leftrightarrow (1-\imath\beta)z_{j+1}=0.
\end{align*}
Therefore, $z_{j+1}=0$.
By induction it follows that $z=0$. 
Hence, it cannot be an eigenvector.

		In the case of $p>1$, the matrices $\mathbf K_t$ and $\mathbf M_t$ have more than one off diagonal and such a relation would not follow so easily.
Numerical experiments in Section~\ref{sec:smallEVMtKt} indicate 
that the real part of $\lambda$ is positive 
for the case $p>1$ too.
\end{remark}

\begin{remark}
Let us consider the case $|\Gamma_D|>0$. From Remark~\ref{rem:EV_realPart_1}, Proposition~\ref{prop:theta0} and the continuous dependence of $\alpha$ on $\theta$, we obtain 
that
       $\mathbf K_x + \alpha \mathbf M_x$ 
must be positive for sufficiently small $\theta$. 
\end{remark}
\begin{remark}
Numerical experiments for various values of $\theta,p$ and $h_n$ in Section~\ref{sec:smallEVMtKt} indicate
that the generalized eigenvalues $\lambda_i$ have a positive real part $\alpha$
provided that the real part of the eigenvalues of $\mathbf M_t$ is positive.
Moreover, in the practical implementation, 
one has to compute the eigenvalues $\lambda_i$ anyway.
Therefore,
we always have an a posteriori control on the positivity of $\alpha$.
If it happens that $\alpha\leq0$, than we have to use a smaller $\theta$. 
\end{remark}

\subsection{Diagonalization}
\label{sec:diagonalization}

If the matrix ${\mathbf{M}}_{t}^{-1}{\mathbf{K}}_{t}$ is diagonalizable,  
the eigenvalue decomposition allows us to write 
\begin{align}
\label{equ:ST:diagDecompA}
	{\mathbf{M}}_{t}^{-1}{\mathbf{K}}_{t}= \mathbf X^{-1} \mathbf D \mathbf X,
\end{align}
where $\mathbf D = \text{diag}(\lambda_i),\, \lambda_i\in\mathbb{C}$, is a diagonal matrix with possibly 
complex eigenvalues on the diagonal, 
and $\mathbf X \in \mathbb{C}^{N_t\times N_t}$ denotes the matrix of the possibly complex eigenvectors. 
Due to the fact that the matrix ${\mathbf{M}}_{t}^{-1}{\mathbf{K}}_{t}$ is non-symmetric, the eigenvectors do not form an orthogonal basis, i.e. $X^{-1} \neq X^H$. 
An efficient calculation can be performed by means of solving the generalized eigenvalue problem $\mathbf K_t x = \lambda \mathbf M_t x$. 

Thanks to
\eqref{equ:ST:diagDecompA}, the matrix $(\mathbf Z\otimes \mathbf M_x + \mathbf I\otimes \mathbf K_x)^{-1}$ from  \eqref{equ:ST:generalDecompA} takes the 
form
\begin{align*}
	(\mathbf Z\otimes \mathbf M_x + \mathbf I\otimes \mathbf K_x)^{-1} = (\mathbf D\otimes \mathbf M_x + \mathbf I\otimes \mathbf K_x)^{-1} = \text{diag}_{i=1,\ldots,N_t}((\mathbf K_x + \lambda_i \mathbf M_x)^{-1}).
\end{align*}
Therefore, only $N_t$ problems of the form $(\mathbf K_x + \lambda_i \mathbf M_x)$ have to be solved, independently of each other. We have to distinguish two cases: the first case where the eigenvalue $\lambda_i$ is a positive real number, and the second one where $\lambda_i$ is a complex number.

In the first case, we consider $\lambda_i = \alpha_i\in \mathbb R^{+}$. In this case the matrix $\mathbf K_x + \lambda_i \mathbf M_x$ is symmetric positive definite. This allows for many possible exact and inexact solution strategies, e.g., Multigrid, Domain Decomposition type methods. 

The second case, where $\lambda_i = \alpha + \imath\beta \in \mathbb{C}$ with $\alpha,\beta\in\mathbb{R}, \alpha>0$, is more difficult to handle. 
We note that $(\mathbf K_x + \lambda_i \mathbf M_x)^H\neq\mathbf K_x + \lambda_i \mathbf M_x$. 
Separating the real and imaginary parts,  we can rewrite 
the complex system $(\mathbf K_x + \lambda_i \mathbf M_x)z = h$ as a real system 
with a real block system matrix of twice  size.
\begin{align*}
	(\mathbf K_x + \lambda_i \mathbf M_x)z &= h\\ \Longleftrightarrow \MatTwo{\mathbf{K}_x + \alpha_i \mathbf{M}_x } {-\beta \mathbf M_x}{\beta \mathbf M_x }{ \mathbf{K}_x + \alpha_i \mathbf{M}_x}
					 \VecTwo{x}{y} &= \VecTwo{f}{g}\\ \Longleftrightarrow 
					\underbrace{\MatTwo{\mathbf{K}_x + \alpha_i \mathbf{M}_x } {\beta_i \mathbf M_x}{\beta_i \mathbf M_x }{-( \mathbf{K}_x + \alpha_i \mathbf{M}_x)}}_{=:\overline{A}_i} \VecTwo{x}{-y} &= \VecTwo{f}{g},
\end{align*}
where $z = x+\imath y$ and $h = f+\imath g$. 
The matrix $\overline{A}_i\in \mathbb{R}^{2N_x\times 2N_x}$ is symmetric, but indefinite. 
We are now looking for an robust preconditioner for $\overline{A}_i$. In order to construct such a preconditioner, we use operator interpolation technique, see, e.g., \cite{HLN:Zulehner:2011a}, \cite{HLN:BerghLofstrom:1976a} and \cite{HLN:AdamsFournier:2003a}. First, we need the definition of the geometric mean of two operators and the general operator interpolation theorem, see also Definition.~2.28 and Theorem.~2.29 in \cite{HLN:MonikaDiss}.
\begin{definition}
	Let $A$ and $B$  be real, symmetric and positive definite matrices. We define the geometric mean of $A$ and $B$ by the relation
	\begin{align*}
		[A,B]_{1/2} = A^{1/2}\big(A^{-1/2}BA^{-1/2}\big)^{1/2}A^{1/2}.
	\end{align*}
	Moreover, for any $\vartheta\in[0,1]$, we define the symmetric and positive matrix by
	\begin{align*}
		[A,B]_{\vartheta} = A^{1/2}\big(A^{-1/2}BA^{-1/2}\big)^{\vartheta}A^{1/2}.
	\end{align*}
\end{definition}

\begin{theorem}
\label{thm:OperatorInterpolation}
	Let $\mathcal{A}:\mathbb{R}^n\to\mathbb{R}^n$ 
	such that the inequalities
	\begin{align*}
		\underline{c}_0\|u\|_{X_0} \leq \|\mathcal{A} u\|_{Y_0}\leq \overline{c}_0\|u\|_{X_0} \quad\text{and}\quad \underline{c}_1\|u\|_{X_1} \leq \|\mathcal{A} u\|_{Y_1}\leq \overline{c}_1\|u\|_{X_1}\quad \forall u\in\mathbb{R}^n
	\end{align*}
	hold,
	where the linear vector spaces $X_j=\mathbb{R}^n$ and $Y_j=\mathbb{R}^n$ with $j\in\{0,1\}$ are equipped with the norms $\|\cdot\|_{X_j}$ and $\|\cdot\|_{Y_j}$, which are associated to the inner products
	\begin{align*}
		(u,v)_{X_j} = (M_ju,v)_{\ell_2} \quad\text{and}\quad (u,v)_{Y_j} = (N_ju,v)_{\ell_2},
	\end{align*}
	given by the symmetric and positive definite matrices $M_0,M_1,N_0$ and $N_1$, and the euclidean inner product  $(\cdot,\cdot)_{\ell_2}$.
	Then, for 
	$X_\vartheta =[X_0,X_1]_\vartheta$ and $Y_\vartheta=[Y_0,Y_1]_\vartheta$,
	with $\vartheta\in[0,1]$, 
	the inequalities
	\begin{align}
		\label{equ:interpolThm}
		\underline{c}_0^{1-\vartheta}\underline{c}_1^\vartheta\|u\|_{X_\vartheta} \leq\|\mathcal{A}u\|_{Y_\vartheta}\leq \overline{c}_0^{1-\vartheta}\overline{c}_1^\vartheta\|u\|_{X_\vartheta} \quad \forall u\in\mathbb{R}^n.
	\end{align}
	hold, where the norms $\|\cdot\|_{X_\vartheta}$ and $\|\cdot\|_{Y_\vartheta}$ are the norms associated to the inner products
	\begin{align*}
		(u,v)_{X_\vartheta}&=(M_\vartheta u,v)_{\ell_2}, \quad\text{with}\quad M_\vartheta = [M_0,M_1]_\vartheta, \quad\text{and}\\
		(u,v)_{Y_\vartheta}&=( N_\vartheta u,v)_{\ell_2}, \quad\text{with}\quad N_\vartheta = [N_0,N_1]_\vartheta,
	\end{align*}
	respectively.
\end{theorem}
\begin{proof}
	For the proof, we refer to the proof of Theorem~2.29 in \cite{HLN:MonikaDiss} and references therein, see also \cite{HLN:AdamsFournier:2003a}.
	 	$\hfill \square$
\end{proof}
\begin{remark}
Using the notation from Theorem~\ref{thm:OperatorInterpolation},  one can show the 
alternative representation 
\begin{align*}
	\|u\|_{X_{\vartheta}}^2 = \frac{2\sin(\vartheta\pi)}{\pi}
	\int_{0}^{\pi}t^{-(2\vartheta+1)}K(t;u)^2\,dt
\end{align*}
of $\|u\|_{X_{\vartheta}}$,
where $K(t;x) = \inf_{x=x_0+x_1}(\|x_0\|_{X_0}^2+t^2\|x_1\|_{X_1}^2)^{1/2}$. From this representation, one observes that 
 \begin{align}
 \label{equ:symmAverage}
  [X_0,X_1]_{\vartheta} = [X_1,X_0]_{1-\vartheta}.
\end{align}
\end{remark}

Let us consider a general saddle point matrix
	\begin{align*}
		\mathcal{A} = \MatTwo{A}{B}{B^T}{-C},
	\end{align*}
	where $A$ and $C$ are symmetric positive definite matrices. We can define two possible negative Schur complements
	\begin{align}
	\label{equ:generalSchurComplements}
		S := C + BA^{-1}B^T \quad \text{and}\quad R :=A + BC^{-1}B,
	\end{align}
	and the associated block diagonal preconditioners 
	\begin{align*}
		P_0 = \MatTwo{A}{0}{0}{S} \quad \text{and} \quad P_1 = \MatTwo{R}{0}{0}{C}.
	\end{align*}
		For $P_0$ and $P_1$, the following spectral inequalities are known
	\begin{align*}
      	(\sqrt{5} -1)/2 \|u\|_{P_j}\leq \|\mathcal{A}u\|_{P_j^{-1}}\leq (\sqrt{5} +1)/2\|u\|_{P_j} \quad j\in\{0,1\},
        \end{align*}
        see Theorem 2.26 in \cite{HLN:MonikaDiss} and references therein. 
	Based on these two preconditioners, we construct a preconditioner $P_\vartheta$ with $\vartheta=1/2$ by an interpolation of the preconditioners $P_0$ and $P_1$:
	\begin{align*}
		P_{1/2} = [P_0,P_1]_{1/2} = 
		\MatTwo{[A,R]_{1/2}}{0}{0}{[S,C]_{1/2}}.
	\end{align*}
By means of Theorem~\ref{thm:OperatorInterpolation} and the setting $M_0=P_0,M_1=P_1,N_0=P^{-1}_0$ and $N_1=P^{-1}$, it follows that
\begin{align*}
	  (\sqrt{5} -1)/2\|u\|_{P_{1/2}}\leq\|\overline{A}u\|_{P^{-1}_{1/2}}\leq  (\sqrt{5} +1)/2\|u\|_{P_{1/2}}.
\end{align*}
Hence, $\text{cond}_{P_{1/2}}(P^{-1}_{1/2}\mathcal{A})\le (\sqrt5+1)/(\sqrt5-1)$.	
	Note, this condition number estimate would hold for all $\vartheta\in[0,1]$.
	In the following, we are looking for an approximation of $P_{1/2}$,
	which can easily be realized in an implementation.
	
\begin{theorem}
\label{thm:ST:PrecondDiag}
	Let  $K_x$ and $M_x$ be symmetric and positive matrices,
	and let $\alpha$ and $\beta$ be real numbers  with $\alpha >0$. 
	Furthermore, we define the block matrices 
	\begin{align}
		\overline{A} &:= \MatTwo{\mathbf{K}_x + \alpha \mathbf{M}_x } {\beta \mathbf M_x}{\beta \mathbf M_x }{-( \mathbf{K}_x + \alpha \mathbf{M}_x)}, \\
					P &:= 
				 \MatTwo{\mathbf{K}_x + (\alpha+|\beta|) \mathbf{M}_x }{0}{0}{\mathbf{K}_x +  (\alpha+|\beta|) \mathbf{M}_x}.
	\end{align}
	Then 
	the condition number estimate
	\begin{align}
		\text{cond}_{P}(P^{-1}\overline{A}) \leq \sqrt2\frac{\sqrt5+1}{\sqrt5-1}
	\end{align}
	holds.
\end{theorem}
\begin{proof}
The proof follows the lines in \cite{HLN:MonikaDiss}, Section 3.3. 
For simplicity, we introduce the notations
$\mathcal K := \mathbf{K}_x + \alpha \mathbf{M}_x$ and $\mathcal M := \mathbf M_x$. Recall the system matrix
	\begin{align*}
		\overline{A} &:= \MatTwo{\mathbf{K}_x + \alpha \mathbf{M}_x } {\beta \mathbf M_x}{\beta \mathbf M_x }{-( \mathbf{K}_x + \alpha \mathbf{M}_x)} = \MatTwo{\mathcal{K} } {\beta \mathcal{M}}{\beta \mathcal{M} }{-\mathcal K}.
	\end{align*}
	Since $\mathcal K$ is symmetric and, due to $\alpha>0$, also positive definite, 
	we can reformulate the two Schur complements from \eqref{equ:generalSchurComplements} for the matrix $\overline A$ as follows:
	\begin{align*}
		S = R =\mathcal{K} + \beta^2\mathcal M \mathcal K^{-1}\mathcal M.
	\end{align*}
	We are looking for an spectral equivalent approximation $P$ of $P_{1/2}$, which is easy to realize and fulfils the spectral inequalities
	\begin{align}
	\label{equ:equivalencePP12}
		\underline{c}P\leq P_{1/2} \leq \overline{c}P,
	\end{align}
	where the constants $\underline{c}$ and $\overline{c}$ are independent of $\alpha$ and $\beta$. Next we estimate $[\mathcal K,R]_{1/2}$ and $[S,\mathcal K]_{1/2}$. 
	Here we make use of the following matrix inequalities
      \begin{align}
      	\frac{1}{\sqrt{2}} (\sqrt a I+ \sqrt b X^{1/2}) \leq (a I + b X)^{1/2} \leq \sqrt a I + \sqrt b X^{1/2},
      \end{align}
      where $X$ is a symmetric positive definite matrix, and $I$ denotes the identity matrix. 
      First we derive an
      upper bound for $[\mathcal K,R]_{1/2}$:
      \begin{align*}
      	[\mathcal K,R]_{1/2} &= \mathcal K^{1/2}\big(\mathcal K^{-1/2}R\mathcal K^{-1/2}\big)^{1/2}\mathcal K^{1/2}  \\
      	& = \mathcal K^{1/2}\big(\mathcal K^{-1/2}(\mathcal{K} + \beta^2\mathcal M \mathcal K^{-1}\mathcal M)\mathcal K^{-1/2}\big)^{1/2}\mathcal K^{1/2}\\
      	&= \mathcal K^{1/2}\big( I + \beta^2\mathcal K^{-1/2}\mathcal M \mathcal K^{-1}\mathcal M \mathcal K^{-1/2}\big)^{1/2}\mathcal K^{1/2}\\
      	&\leq \mathcal K^{1/2}\big( I + (\beta^2\mathcal K^{-1/2}\mathcal M \mathcal K^{-1}\mathcal M \mathcal K^{-1/2})^{1/2}\big)\mathcal K^{1/2}\\
      	&= \mathcal K + |\beta| \mathcal K^{1/2}(\mathcal K^{-1/2}\mathcal M \mathcal K^{-1}\mathcal M \mathcal K^{-1/2})^{1/2}\mathcal K^{1/2}\\
      	&= \mathcal K + |\beta| \mathcal K^{1/2}(\mathcal K^{-1/2}\mathcal M \mathcal K^{-1/2})^{1/2}(\mathcal K^{-1/2}\mathcal M \mathcal K^{-1/2})^{1/2}\mathcal K^{1/2}\\
      	&=\mathcal K + |\beta| \mathcal K^{1/2}(\mathcal K^{-1/2}\mathcal M \mathcal K^{-1/2})\mathcal K^{1/2}\\
	&= \mathcal K + |\beta| \mathcal M.\\
      \end{align*}
      Similarly, for the lower bound, we obtain
            \begin{align*}
      	[\mathcal K,R]_{1/2} &= \mathcal K^{1/2}\big(\mathcal K^{-1/2}R\mathcal K^{-1/2}\big)^{1/2}\mathcal K^{1/2}  \\
      	&= \mathcal K^{1/2}\big( I + \beta^2\mathcal K^{-1/2}\mathcal M \mathcal K^{-1}\mathcal M \mathcal K^{-1/2}\big)^{1/2}\mathcal K^{1/2}\\
      	&\geq \mathcal K^{1/2}  \big( \frac{1}{\sqrt{2}}(I + (\beta^2\mathcal K^{-1/2}\mathcal M \mathcal K^{-1}\mathcal M \mathcal K^{-1/2})^{1/2})\big)\mathcal K^{1/2}\\
      	&= \frac{1}{\sqrt{2}} (\mathcal K + |\beta| \mathcal K^{1/2}(\mathcal K^{-1/2}\mathcal M \mathcal K^{-1}\mathcal M \mathcal K^{-1/2})^{1/2}\mathcal K^{1/2})\\
	&= \frac{1}{\sqrt{2}} (\mathcal K + |\beta| \mathcal M).\\
      \end{align*}
      The missing estimate from above and below for $[S,\mathcal K]_{1/2}$  follow from the fact that $[S,\mathcal K]_{1/2}= [\mathcal K,S]_{1/2}=[\mathcal K,R]_{1/2}$, see \eqref{equ:symmAverage}. Hence, for the preconditioner
      \begin{align*}
      	P:=\MatTwo{\mathcal K + |\beta| \mathcal M}{0}{0}{\mathcal K + |\beta| \mathcal M} = \MatTwo{\mathbf{K}_x + (\alpha+|\beta|) \mathbf{M}_x }{0}{0}{\mathbf{K}_x +  (\alpha+|\beta|) \mathbf{M}_x},
      \end{align*}
      we obtain the spectral constants $\underline{c} = \frac{1}{\sqrt{2}}$ and $\overline{c} = 1$ in \eqref{equ:equivalencePP12}. Finally, we arrive at the estimate
     \begin{align}
      \label{equ:conditionnumberestimate}
      	\text{cond}_{P}(P^{-1}\overline{A})= \|P^{-1}\overline{A}\|_P\|\overline{A}^{-1}P\|_P \leq \sqrt2\|P^{-1}_{1/2}\overline{A}\|_{P_{1/2}}\|\overline{A}^{-1}P_{1/2}\|_{P_{1/2}}\leq \sqrt2\frac{\sqrt5+1}{\sqrt5-1} .
      \end{align}
       	$\hfill \square$
      \end{proof}

      \begin{remark}
      \label{rem:cond_complexEV}
      	The estimate \eqref{equ:conditionnumberestimate} of the condition number 
      	$\text{cond}_{P}(P^{-1}\overline{A})$ can be improved by solving 
      	the generalized eigenvalue problem
      	\begin{align*}
      		\overline{A}\VecTwo{x}{y} = \lambda P\VecTwo{x}{y}
      	\end{align*}
      	directly.
      	Following the procedure outlined in Remark~9 in \cite{HLN:Zulehner:2011a}, see also the proof of Theorem~\ref{thm:ST:PrecondRealSchur}, we find that the generalized eigenvalues 
      	satisfy the estimates
      	\begin{align*}
      		|\lambda_{min}|\geq\frac{1}{\sqrt{2}}\quad \text{and}\quad |\lambda_{max}|\leq 1,
      	\end{align*}
         which leads to 
         the condition number estimate
         $\text{cond}_P(P^{-1}A) \leq \sqrt{2}$.
      \end{remark}

We note that both block-diagonal entries of $P$ are identical, and the matrix $\mathbf{K}_x + (\alpha+|\beta|) \mathbf{M}_x$ is symmetric and positive definite. 
This opens various possibilities for preconditioning 
based on standard techniques for symmetric and positive definite matrices.
The linear system $\overline{A}y = f$ can then be solved, e.g., by means of MinRes preconditioned by $P^{-1}$. We can even use an spectral equivalent approximation $\hat{P}^{-1}$, i.e., $c\hat{P}^{-1}\leq P^{-1}\leq C\hat{P}^{-1}$, with constants $c$ and $C$, independent of $\alpha$ and $\beta$. Moreover, this approach allows for a further parallelization by applying $\mathbf A_n$ in parallel for $n=1,\ldots,N_t$. 

Unfortunately, this approach has a severe drawback. 
Due to the fact that the matrix ${\mathbf{M}}_{t}^{-1}{\mathbf{K}}_{t}$ is non-symmetric, the matrix $\mathbf X$ of eigenvectors is not unitary and, therefore, $\text{cond}(\mathbf X)\neq1$. Actually, numerical tests in \secref{sec:ST_NumTests_X} show that, for large B-Spline degree or small $h_t$, we observe that the condition number $\text{cond}(\mathbf X)\approx 10^{12}$. In that case we cannot correctly apply \eqref{equ:ST:generalDecompA} and the algorithm fails. This problem can be circumvented by using the Complex or Real Schur decomposition, as presented in the subsequent two subsections.

\subsection{Complex Schur decomposition}
\label{sec:complexSchur}
In this section, we investigate an alternative possibility for decomposing ${\mathbf{M}}_{t}^{-1}{\mathbf{K}}_{t}$. The Complex Schur decomposition provides a decomposition of the form 
\begin{align}
\label{equ:ST:CSchurDecompA}
	{\mathbf{M}}_{t}^{-1}{\mathbf{K}}_{t}= \mathbf Q^{*} \mathbf T \mathbf Q,
\end{align}
where $\mathbf Q\in\mathbb{C}^{N_t\times N_t}$ and $\mathbf T\in \mathbb{C}^{N_t\times N_t}$ is a upper triangular matrix with $T_{ii} = \lambda_i$. The advantage of the  (complex) Schur decomposition is the fact that we obtain a unitary matrix $\mathbf Q$. 
Hence, $\text{cond}(\mathbf Q) = 1$, 
but the diagonal matrix $\mathbf D$ in the decomposition \eqref{equ:ST:diagDecompA} 
is now replaced by the upper triangular matrix $\mathbf T$ in the decomposition \eqref{equ:ST:CSchurDecompA},
By means of \eqref{equ:ST:CSchurDecompA}, the matrix $\mathbf Z\otimes \mathbf M_x + \mathbf I\otimes \mathbf K_x$ from  \eqref{equ:ST:generalDecompA} takes the 
form
\begin{align*}
	(\mathbf Z\otimes \mathbf M_x + \mathbf I\otimes \mathbf K_x)^{-1} &= (\mathbf T\otimes \mathbf M_x + \mathbf I\otimes \mathbf K_x)^{-1} \\
	&= 	\begin{bmatrix}
		\mathbf K_x+ T_{11} \mathbf M_x & T_{12} \mathbf M_x &\ldots &\\
		0 & \mathbf K_x+ T_{22} \mathbf M_x & T_{23} \mathbf M_x& \\
		\vdots &0 & \ddots &               T_{N_t N_t-1} \mathbf M_x \\
		0      &\ldots &     0   & \mathbf K_x+ T_{N_t N_t} \mathbf M_x  
	\end{bmatrix}^{-1}\\
	&= 	\begin{bmatrix}
		\mathbf K_x+ \lambda_1 \mathbf M_x & T_{12} \mathbf M_x &\ldots &\\
		0 & \mathbf K_x+ \lambda_2 \mathbf M_x & T_{23} \mathbf M_x& \\
		\vdots &0 & \ddots &               T_{N_t N_t-1} \mathbf M_x \\
		0      &\ldots &     0   & \mathbf K_x+ \lambda_{N_t} \mathbf M_x  
	\end{bmatrix}^{-1}
\end{align*}

The application of 
$(\mathbf T\otimes \mathbf M_x + \mathbf I\otimes \mathbf K_x)^{-1}$
to some vector $f$ can be performed staggered way 
as
presented in Algorithm~\ref{alg:ApplComplexSchur}.

\begin{algorithm}
	\caption{Calculation of $y = (\mathbf T\otimes \mathbf M_x + \mathbf I\otimes \mathbf K_x)^{-1}f$}
   \label{alg:ApplComplexSchur}
   \begin{algorithmic}
   \algblock{Begin}{End}
    \For {$i =N_t,N_t-1\ldots,1$}
        \State $g = f_i$
        \For {$j =i+1,i+2\ldots,N_t$}
	  \State  $g = g- T_{ij}y_j$
	\EndFor
	\State Solve $(\mathbf K_x+ \lambda_i \mathbf M_x)y_i=g$, where $\lambda_i = T_{ii}$.
    \EndFor
   \State \Return $y$
  \end{algorithmic}
\end{algorithm}

In order to solve the linear systems $(\mathbf K_x+ \lambda_i \mathbf M_x)y_i=g, i=1,\ldots,N_t$ in Algorithm~\ref{alg:ApplComplexSchur}, we can use the techniques developed in the previous subsection. This decomposition method allows us to have a well conditioned transformation matrix $Q$, however at the cost that the linear system cannot be solved independently of each other. 
We note that this method and the eigenvalue decomposition require complex arithmetic, which is more expensive than the real one. In the following subsection, we investigate the real Schur decomposition, which eliminates the need for having complex arithmetic.

\subsection{Real Schur decomposition}
\label{sec:realSchur}

In this subsection, we look at the decomposition of ${\mathbf{M}}_{t}^{-1}{\mathbf{K}}_{t}$ by means of the Real Schur decomposition. It provides a decomposition of the form 
\begin{align}
\label{equ:ST:RSchurDecompA}
	{\mathbf{M}}_{t}^{-1}{\mathbf{K}}_{t}= \mathbf Q^{*} \mathbf T \mathbf Q,
\end{align}
where $\mathbf Q\in\mathbb{R}^{N_t\times N_t}$. The matrix $\mathbf T\in \mathbb{R}^{N_t\times N_t}$ is a upper quasi-triangular matrix, i.e., the diagonal consists of $1\times1$ and $2\times2$ blocks. The values of the $1\times1$ blocks correspond to the real eigenvalues,
while the $2\times2$ blocks correspond to the complex eigenvalues of ${\mathbf{M}}_{t}^{-1}{\mathbf{K}}_{t}$. 

By additionally performing a Givens rotation, the $2\times2$ block can be transformed 
to
the structure
 \begin{align*}
 	\mathbf B:= \MatTwo{\alpha}{\beta_1}{\beta_2}{\alpha},
 \end{align*}
 where $\alpha,\beta_1,\beta_2\in\mathbb{R}$ and $\beta_1\neq\beta_2\neq0$. The eigenvalues of this matrix are given by $\alpha \pm\sqrt{\beta_1\beta_2}$. Due to the fact that the eigenvalues have to be complex and the real part has to be positive, we obtain that $\alpha>0$ and $\beta_1$ and $\beta_2$ have different signs. Therefore, we can write the eigenvalues as $\alpha \pm\imath\sqrt{|\beta_1\beta_2|}$.

Using this decomposition, the matrix $\mathbf Z\otimes \mathbf M_x + \mathbf I\otimes \mathbf K_x$ appearing in \eqref{equ:ST:generalDecompA} has a structure, which is similar to that one of the Complex Schur decomposition. 
The corresponding system of linear algebraic equations
can also be again solved in a staggered way as presented in Algorithm~\ref{alg:ApplComplexSchur}. 
One has to adapt the algorithm in such a way that, if the diagonal block is a $2\times2$ block, one has to work with two-block vectors and a $2\times2$ block matrix. 
It remains to investigate the solution strategy for the $2\times2$ block matrix. As already mentioned, the $2\times2$ block of $T$ is non-symmetric. 
Hence, the $2\times2$ block matrix is also non-symmetric and is given in the following way
\begin{align*}
	\MatTwo{\mathbf{K}_x + \alpha \mathbf{M}_x } {\beta_1 \mathbf M_x}{\beta_2 \mathbf M_x }{ \mathbf{K}_x + \alpha \mathbf{M}_x}.
\end{align*}
The structure of the matrix is very similar to $\overline{A}$ in Theorem~\ref{thm:ST:PrecondDiag} up to the non-symmetry, which origins just from the different scalings $\beta_1$ and $\beta_2$ and their different sign. By a proper rescaling, we can transform this linear system into an equivalent 
system with a symmetric, but indefinite system matrix:
\begin{align*}
\MatTwo{\mathbf{K}_x + \alpha \mathbf{M}_x } {\beta_1 \mathbf M_x}{\beta_2 \mathbf M_x }{ \mathbf{K}_x + \alpha \mathbf{M}_x}\VecTwo{x}{y} &= \VecTwo{f}{g}\\
	\Longleftrightarrow \MatTwo{\mathbf{K}_x + \alpha \mathbf{M}_x } {-\beta_1 \mathbf M_x}{\beta_2 \mathbf M_x }{ -(\mathbf{K}_x + \alpha \mathbf{M}_x)} \VecTwo{x}{-y} &= \VecTwo{f}{g}\\
	\Longleftrightarrow \underbrace{\MatTwo{|\beta_2|(\mathbf{K}_x + \alpha \mathbf{M}_x) } {-\beta_1|\beta_2| \mathbf M_x}{|\beta_1|\beta_2 \mathbf M_x }{ -|\beta_1|(\mathbf{K}_x + \alpha \mathbf{M}_x)}}_{=:\overline{A}}\VecTwo{x}{-y} &= \VecTwo{|\beta_2|f}{|\beta_1|g},
\end{align*}
We note that $\beta_1$ and $\beta_2$ have different signs. 
Hence, $-\beta_1|\beta_2| = -\beta_2|\beta_1|$. 
Motivated by the construction of the preconditioner in the case of the eigenvalue decomposition, we can come up with an optimal preconditioner. The following theorem presents this optimal preconditioner for the matrix $\overline{A}$. 

\begin{theorem}
\label{thm:ST:PrecondRealSchur}
	Let  $K_x$ and $M_x$ be symmetric and positive matrices,
	and let $\alpha,\beta_1,\beta_2$ be real numbers  with $\alpha >0$. 
	Furthermore, we define the block matrices 
	\begin{align*}
		\overline{A} &:= \MatTwo{|\beta_2|(\mathbf{K}_x + \alpha \mathbf{M}_x) } {-\beta_1|\beta_2| \mathbf M_x}{|\beta_1|\beta_2 \mathbf M_x }{ -|\beta_1|(\mathbf{K}_x + \alpha \mathbf{M}_x)}, \\
					P &:= 
				 \MatTwo{|\beta_2|(\mathbf{K}_x + (\alpha+\sqrt{|\beta_1\beta_2|}) \mathbf{M}_x) }{0}{0}{|\beta_1|(\mathbf{K}_x +  (\alpha+\sqrt{|\beta_1\beta_2|}) \mathbf{M}_x)}.
	\end{align*}
	Then 
	the condition number estimate
	\begin{align*}
		\text{cond}(P^{-1}\overline{A}) \leq \sqrt{2}.
	\end{align*}
	holds.
\end{theorem}
\begin{proof}
The proof follows the lines from Remark~9 in \cite{HLN:Zulehner:2011a}, which gives a sharper bound than using interpolation theory as in \cite{HLN:MonikaDiss}.
For notational simplicity, we introduce the abbreviations $\mathcal K := \mathbf{K}_x + \alpha \mathbf{M}_x$ and $\mathcal M := \mathbf M_x$. 
We now consider  the generalized eigenvalue problem $\overline{A} u = \lambda P u$, which reads
\begin{align}
\label{equ:GenEigenvalueProblem}
\MatTwo{ |\beta_2|\mathcal K } {-\beta_1|\beta_2| \mathcal M}{|\beta_1|\beta_2 \mathcal M }{ -|\beta_1|\mathcal K}\VecTwo{x}{y} = \lambda
			 \MatTwo{|\beta_2|(\mathcal{K} + \sqrt{|\beta_1\beta_2|} \mathcal{M}) }{0}{0}{|\beta_1|(\mathcal{K} + \sqrt{|\beta_1\beta_2|}\mathcal{M})}\VecTwo{x}{y}.
\end{align}
At first we consider the generalized eigenvalue problem
\begin{align*}
	\mathcal K z = \mu (\mathcal{K} + \sqrt{|\beta_1\beta_2|}\mathcal M)z.
\end{align*}
Due to the fact that $\mathcal K$ and $\mathcal M$ are symmetric, there exists an basis $\{e_1,e_2,\ldots,  e_{N_x}
\}$ of eigenvectors, which are orthonormal with respect to the inner product generated by $\mathcal{K} + \sqrt{|\beta_1\beta_2|}\mathcal M$, and 
corresponding eigenvalues  $\mu_j$. Since $\mathcal K$ is dominated by $\mathcal K + \sqrt{|\beta_1\beta_2|}\mathcal M$ and due to their positivity, we have that $\mu_j\in[0,1]$.
Therefore, we can express $x$ and $y$ as linear combination of $e_j$ with coefficients $\hat{x}_j$ and $\hat{y}_j$, respectively. Moreover, $\mathcal M z$ fulfils the following identity
\begin{align*}
	\mathcal Mz &= (|\beta_1\beta_2|)^{-1/2}(\sqrt{|\beta_1\beta_2|}\mathcal M + \mathcal K)z - (|\beta_1\beta_2|)^{-1/2} \mathcal K z\\
		    &= (|\beta_1\beta_2|)^{-1/2}(\sqrt{|\beta_1\beta_2|}\mathcal M + \mathcal K)z - (|\beta_1\beta_2|)^{-1/2} \mu (\mathcal{K} + \sqrt{|\beta_1\beta_2|}\mathcal M)z\\
		    &= (|\beta_1\beta_2|)^{-1/2}(1-\mu)(\sqrt{|\beta_1\beta_2|}\mathcal M + \mathcal K)z.
\end{align*}
Using the expansion of $x$ and $y$ into the eigenvectors $\{e_j\}$, 
system \eqref{equ:GenEigenvalueProblem} decomposes into the $2\times 2$ systems
\begin{align*}
\MatTwo{|\beta_2|\mu_j}{-\beta_1|\beta_2||\beta_1\beta_2|^{-1/2}(1-\mu_j)}{ |\beta_1|\beta_2|\beta_1\beta_2|^{-1/2}(1-\mu_j)}{-|\beta_1|\mu_j}\VecTwo{\hat{x}_j}{\hat{y}_j} = \lambda \MatTwo{|\beta_2|}{0}{0}{|\beta_1|}\VecTwo{\hat{x}_j}{\hat{y}_j}.
\end{align*}
Since there exists at least one pair $(\hat{x}_j,\hat{y}_j)$ which is non-zero, 
the determinant of the system matrix must be zero, i.e.,
\begin{align*}
\text{det}\left(\MatTwo{|\beta_2|\mu_j}{-\beta_1|\beta_2||\beta_1\beta_2|^{-1/2}(1-\mu_j)}{ |\beta_1|\beta_2|\beta_1\beta_2|^{-1/2}(1-\mu_j)}{-|\beta_1|\mu_j}- \lambda \MatTwo{|\beta_2|}{0}{0}{|\beta_1|}\right)=0,
\end{align*}
which reduces to
\begin{align*}
	|\beta_1\beta_2|(\lambda^2 - \mu_j^2) - (\beta_1\beta_2)^2 |\beta_1\beta_2|^{-1}(1-\mu_j)^2 = 0,
\end{align*}
where we used that $-\beta_1|\beta_2| = |\beta_1|\beta_2\neq0$. We immediately obtain that $|\lambda| = \sqrt{\mu_j^2 + (1-\mu_j)^2}$ for $\mu_i\in[0,1]$ and it follows that $\frac{1}{\sqrt2}\leq|\lambda|\leq 1$, which gives the desired bound on the condition number 
of $P^{-1}\overline{A}$.
 	$\hfill \square$
\end{proof}
Now we can again use the MinRes preconditioned by $P$ as iterative solver for systems
with the system matrix $\overline{A}$,
and we obtain a robust method. Moreover, due to the use of real arithmetic, this approach is usually more efficient than that one using the Complex Schur decomposition.

\section{Numerical examples}
\label{sec:numerics}

In this section, we test the proposed preconditioners on the three (2+1) dimensional space-time
cylinder $Q$
illustrated in Figure~\ref{fig:SpaceTimeYeti}. 
The two dimensional spatial domain $\Omega$ consists of 21 spatial subdomains 
(volumetric patches). 
For each time slap, we use conforming B-Splines of degree $p$. 
The problems were calculated on a 
Desktop PC with an Intel(R) Xeon(R) CPU E5-1650 v2 @ 3.50GHz and 16 GB main memory. 
We use the C++ library G+Smo for describing the geometry and performing the numerical tests,  see also \cite{HLN:JuettlerLangerMantzaflarisMooreZulehner:2014a} and \cite{gismoweb}.
\begin{figure}
	\includegraphics[width=0.3\textwidth]{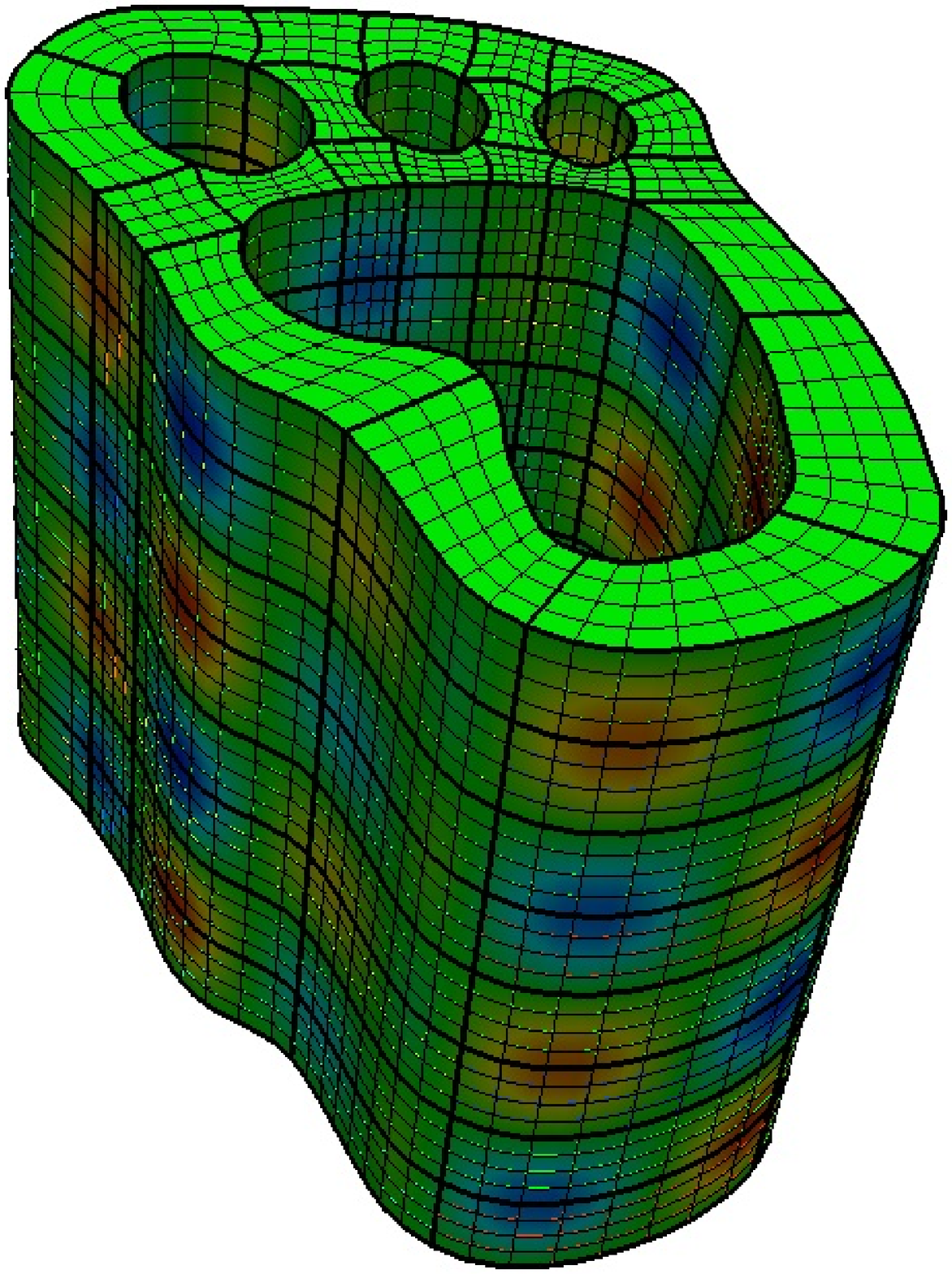}   
	\hspace{10ex}
	\includegraphics[width=0.25\textwidth]{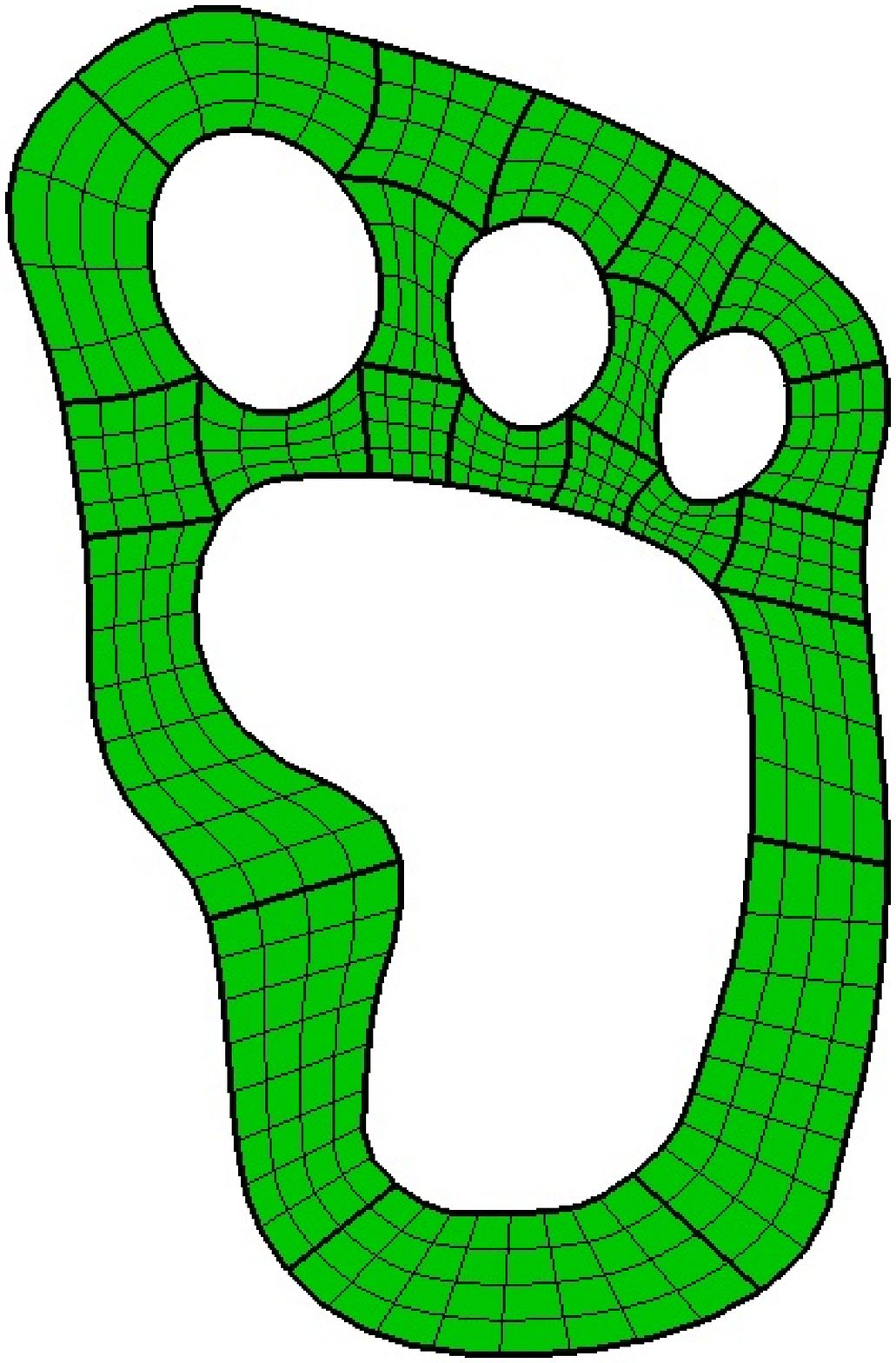}
	\begin{normalsize}
	\caption{
	The left picture shows the space-time cylinder $Q$ with 8 time slabs, while the right picture presents the spatial domain $\Omega$ consisting of 21 patches.}
	    \label{fig:SpaceTimeYeti}
	    \end{normalsize}
\end{figure}

\subsection{Condition number of eigenvector matrix $X$}
\label{sec:ST_NumTests_X}
Here, we study the condition number of the generalized eigenvectors of $(\mathbf K_t,\mathbf M_t)$. Due to the non-symmetry of $K_t$ and $M_t$, we do not obtain an orthogonal basis of eigenvectors. Hence, the condition number is not $1$. 
Actually, it can be quite large. 
We report on the condition number for different $p$ and 
$N_t$
in Table~\ref{tab:condX}. We observe that the condition number grows exponentially with $p$ and $N_t$.  We conclude that for small $p$ or small number of dofs in time direction, the approach presented in Section~\ref{sec:diagonalization} may be still feasible.

	\begin{table}
	 \centering
	   	  \begin{normalsize}
\begin{tabular}{|r||r|r|r|r|r|r|r|r|}\hline
$N_t-p$ \,\textbackslash\, $p$   &    $2$      &    $3$   &    $4$   &    $5$   	&    $6$   &    $7$   &    $8$  \\ \hline \hline
2       &    64   	& 309  	     & 362   		& 766   & 1706 	 & 3907  & 9501    \\ \hline
4       &    481   	& 1036   	& 3037    	& 9419  & 41959  & 39323  & 73946       \\ \hline
8      &    2869    	& 16118   	& 39693    	& 74370 & 180054 & 472758 & 1e+06 \\ \hline
16      &    34332    	& 188263     & 463148     	& 1e+06	& 6e+06	& 3e+07 & 1e+08 \\ \hline
32      &    701306     & 2e+06		& 1e+07		& 6e+07	& 4e+08 & 7e+09 & 1e+10 \\ \hline
64      &    5e+07	& 4e+07		& 3e+08		& 3e+09	& 6e+10	& 3e+11 & 1e+12 \\ \hline
128     &    2e+08	& 1e+09		& 1e+10		& 3e+11	& 2e+13	& 5e+13 & 4e+14  \\ \hline
\end{tabular}

   \caption{Condition number of $\mathbf X$ for $\theta = 0.01$ and $|t_{n+1}-t_n|=0.1$.}
   \label{tab:condX}
             \end{normalsize}
   \end{table}

\subsection{Smallest eigenvalue of $\mathbf M_t^{-1} \mathbf K_t$}
\label{sec:smallEVMtKt}
In Section~\ref{sec:constructionAinv}, we observed the necessity that  the real part of the smallest eigenvalue of  $\mathbf M_t^{-1} \mathbf K_t$ is positive. In this section, we present numerical studies for different $p$, $h$ and $\theta$, where we fix the time interval to $[0,1]$. The results are summarized in Table~\ref{tab:minEV}, where the entries with $*$ indicate that the matrix $M_t$ had at least one eigenvalue with negative real part. Consequently, the smallest real part of the generalized eigenvalues was also  negative. We observe that, if $\mathbf M_t>0$, then also the real part of  $\mathbf M_t^{-1} \mathbf K_t$ is positive. The positive real part of the eigenvalues for the $p=1$ and $\theta=0$ is in agreement with Remark~\ref{rem:EV_realPart_2}. Moreover, for $\theta=0$ and increasing $p$ we observe even an increase of the smallest real part of the eigenvalues, cf. Proposition~\ref{prop:theta0} and Remark~\ref{rem:EV_realPart_2}. 
The numerical tests indicate that, for sufficiently small $\theta$, the smallest real part of the generalized eigenvalues stays positive.

\begin{table}
	 \centering
	   	  \begin{normalsize}
	\begin{tabular}{|r||r|r|r|r|r|r|r||r|r|r|r|r|r|r|}\hline
	&\multicolumn{7}{c||}{2 uniform refinements}  &    \multicolumn{7}{c|}{4 uniform refinements}    \\\hline
$\theta \textbackslash p  $   & 1 &   2      &   3  &    4  &    5   	&   6  &   7  & 1  &   2      &   3  &    4  &    5   	&   6  &   7 \\ \hline
0       & 1.5 &    2.4   & 3.2     & 3.8   & 4.3  &  4.7   & 5.0            & 0.2   & 0.5   & 0.9     & 1.5   & 2.1      &  2.7      & 3.4         \\ 
0.01    & 1.6 &    2.5   & 3.2     & 3.6   & 4.0  &  4.4   & 4.9            & 0.7   & 0.7   & 1.1     & 1.6   & 2.2      &  2.8      & 3.3         \\ 
0.1     & 2.5 &    2.9   & 3.2     & 3.6   & 4.0  &  4.5   & 5.2            & 4.8   & 2.9   & 2.7     & 3.0   & 3.4      &  3.6      & 4.1        \\ 
1       & 4.1 &    4.5   & 4.7     & *     & *     &  *      & *            & 12.4   & 12.0   & 9.2     &  *    & *         &  *       &  *          \\
10      & 4.6 &    5.2   & 5.2     & *     & *     &  *      & *            & 6.7    & 11.8   & *       &  *    & *         &  *       & *          \\ \hline \hline
	&\multicolumn{7}{c||}{6 uniform refinements}  &    \multicolumn{7}{c|}{8 uniform refinements}    \\\hline  
	$\theta \textbackslash p  $   & 1 &   2      &   3  &    4  &    5   	&   6  &   7  & 1  &   2      &   3  &    4  &    5   	&   6  &   7 \\ \hline
0       & 0.01 &   0.03   & 0.06    & 0.1   & 0.1      &  0.2      & 0.2   &  0.0008 & 0.002  & 0.004   & 0.006  &  0.009    &  0.01      & 0.02     \\ 
0.01    & 1.9 &   1.0   & 0.8     & 0.7   & 0.6      &  0.6      & 0.6    &   7.7   &4.0    & 3.0     & 2.5    &  2.0      &  1.8       & 1.6     \\  
0.1     & 18.6 &   9.9   & 7.4     & 6.0   & 5.1      &  4.5      & 4.0    &  34.8 & 33.8   & 29.5    & 23.8   &  20.0     &  17.2      & 15.1    \\   
1       & 34.2 &   35.1   & 33.8    &  *    & *         &  *       &  *    &  34.8 & 34.4   & 34.5    &  *     & *         &  *       &  *      \\ 
10      & 11.4 &   17.4   & *       &  *    & *         &  *       & *     &  29.0 & 32.2   & *       &  *     & *         &  *       & *     \\ \hline   
\end{tabular}
  \caption{Smallest real part of generalized eigenvalues $\mathbf K_tx=\lambda \mathbf M_t x$ for different B-Spline degrees $p$, $\theta$ and number of dofs. The $*$ indicates that the matrix $\mathbf M_t$ has at least one eigenvalue with negative real part.}
     \label{tab:minEV}
          \end{normalsize}
\end{table}

\subsection{Condition number of preconditioned $\mathbf K_x +\lambda \mathbf M_x$}

The aim of this section is to verify the optimal condition number bound 
presented
in Theorem~\ref{thm:ST:PrecondDiag} and Theorem~\ref{thm:ST:PrecondRealSchur}.
To do so, we report on the maximum number of MinRes-iterations in order to solve $\mathbf K_x +\lambda_i \mathbf M_x$, where $\lambda_i\in\mathbb{C}$ are the generalized eigenvalues of $(\mathbf K_t,\mathbf M_t)$. 
We use zero initial guess, and a reduction of the initial residual by $10^{-10}$.
We choose $\theta=0.1$. 
In Table~\ref{tab:precKxMxComplex}, we investigate the robustness of the preconditioners from Theorem~\ref{thm:ST:PrecondDiag} and Theorem~\ref{thm:ST:PrecondRealSchur}. We observe that the number of iterations stays bounded for various $p$ and $h$.

       \begin{table}
       	 \centering
	   	  \begin{normalsize}
    \begin{tabular}{|c||P{0.7cm}|P{0.7cm}|P{0.7cm}|P{0.7cm}|P{0.7cm}||P{0.7cm}|P{0.7cm}|P{0.7cm}|P{0.7cm}|P{0.7cm}|}\hline
    & \multicolumn{5}{c||}{Complex Schur decomp.} &     \multicolumn{5}{c|}{Real Schur decomp.}  \\\hline
    	ref. $x$ and $t$\,\textbackslash\,$p$        &	$2$	&	$3$	&	$4$	& 	$5$	& $6$           &	$2$	&	$3$	&	$4$	& 	$5$	& $6$    \\\hline
    	0	   &      23      &     22          &      26         &    26           &   26    &      18      &     18          &      20         &    21           &   22  \\
    	1          &      25      &     24          &      24         &    27           &   26    &      20      &     20          &      22         &    22           &   22   \\
    	2          &      25      &     25          &      25         &    27           &   27    &      22      &     22          &      22         &    22           &   22  \\
    	3          &      24      &     26          &      26         &    27           &   27    &      22      &     22          &      22         &    22           &   21    \\
    	4          &      25      &     25          &      26         &    27           &   26    &      22      &     22          &      22         &    22           &   20    \\ \hline
    \end{tabular}
    \caption{Maximum number of MinRes iterations to solve $\mathbf K_x +\lambda_i \mathbf M_x$, $i=1,\ldots,N_t$, resulting from the Complex and Real Schur decomposition. Refinement is performed uniformly in $x$ and $t$.}
        \label{tab:precKxMxComplex}
              \end{normalsize}
    \end{table}

 \subsection{Application to Space-Time Multigrid}
This section deals with the use of the iterative methods developed 
in Section~\ref{sec:constructionAinv}
as smoothers in the space-time multigrid.
 The realization of the preconditioner $P$, see Theorem~\ref{thm:ST:PrecondDiag} and Theorem~\ref{thm:ST:PrecondRealSchur} is performed via a sparse direct solver.  
 We 
 use
 the PARDISO 5.0.0 Solver Project \cite{HLN:PARDISO500} for performing the LU factorizations. 
 We compare the three different approaches, presented in Section~\ref{sec:constructionAinv}, 
 with 
 the exact realization of 
 $\mathbf A_n^{-1}$ via 
 the sparse direct solver PARDISO.
 For approximating $\mathbf A_n^{-1}$ via MinRes, we use zero initial guess and a reduction of the initial residuum by $10^{-4}$. In Table~\ref{tab:ST_MG}, we report on the single core computation time of the MG algorithm to setup the data-structures and solve the system via the 
 MG iteration. The setup time includes the LU factorizations, but not the assembling of the matrices. For the MG iteration, we use zero initial guess and a reduction of the initial residuum by $10^{-8}$. We choose $\theta=0.01$, $|t_n-t_{n+1}|=0.1$ and the polynomial degree by $p=3$ for both space and time direction. Moreover, we fix the number of dofs in time direction of a time slab, but increase the number of time slabs. The MG method uses coarsening in space as well as in time. 
 
 We observe that the LU factorization of $\mathbf A_n$ needs a quite large amount of time, 
 whereas 
 the setup time is almost negligible 
 for the three preconditioners proposed. 
 The little increase in the solution time definitely pays off by the small setup time. In addition, the Real-Schur decomposition almost provides the same solution time as the direct solver. Due to the complex arithmetic of the Diagonalization or the Complex-Schur decomposition, their computational effort doubles, which we observe also in the numerical test. Finally, due to the quite accurate approximation of $\mathbf{A}_n^{-1}$ (up to $10^{-4}$), 
 we do not observe a deterioration of the MG iteration numbers. 
 It took around 12 iterations to reach the desired tolerance of $10^{-4}$.
 
 \begin{table}
 	\centering
 	\begin{normalsize}
 		\begin{tabular}{|r|cc|r|r|rr|rr|}     \hline
       $\#$dofs & \multicolumn{2}{c|}{ref} &  $\#$slaps     & MG-It  & \multicolumn{2}{c|}{Direct} &\multicolumn{2}{c|}{Diag}     \\
 	         & x        & t            &                &        &   Setup        & Solving    &    Setup     & Solving        \\   \hline
	 15950  &  2        & 3            &      2         &    7   &    1.9         &    0.7     &  0.04        &    2.3         \\
	97020   &  3        & 3            &      4         &    7   &    38.6        &    8.5     &  0.3         &   19.4         \\
        665720  &  4        & 3            &      8         &    7   &    1008        &   94.6     &  3.7         &  183.8         \\ \hline
       $\#$dofs & \multicolumn{2}{c|}{ref} &  $\#$slaps     & MG-It  & \multicolumn{2}{c|}{C-Schur} &\multicolumn{2}{c|}{R-Schur}  \\   \hline
	 15950  &  2        & 3            &      2         &    7   &    0.05        &     2.4   &   0.04       &   1.3           \\
	97020   &  3        & 3            &      4         &    7   &    0.5         &    19.9   &   0.3        &  11.1           \\
        665720  &  4        & 3            &      8         &    7   &   5.4          &   187.3   &   3.7        & 108.0           \\	 \hline
\end{tabular}
\caption{Comparison of the 
Diagonalization as well as the
Complex Schur and Real Schur decompositions with a sparse direct solver used for approximating $\mathbf A_n^{-1}$. All timings are given in seconds.}
\label{tab:ST_MG}
 	\end{normalsize}
 \end{table}

 \section{Conclusions}
 \label{sec:Conclusions}
 
 In this work, we presented a decomposition of a non-symmetric linear system arising from a space-time formulation into a series of symmetric linear systems, which are easier to solve. These problems are part of the time-parallel MG method introduced in \cite{HLN:Neumueller:2013a}. 
 They correspond to spatial problems. 
 They are either symmetric and positive definite or have a symmetric saddle point structure. 
 For the latter, we presented robust preconditioners motivated by operator interpolation theory. 
 The runtime performance is already very promising, even when using direct solvers,  
 and can further be reduced by using robust IgA multigrid or IgA domain decomposition approaches 
 as proposed, e.g., in 
 \cite{HLN:HoferLanger:2017a,HLN:Hofer:2018a,HLN:HoferLanger:2019b} 
 or
 \cite{HLN:HofreitherTakacsZulehner:2017a,HLN:HofreitherTakacs:2017a},
 respectively. 
 The advantage of the  decompositions proposed 
 consists in the availability of
 well-established preconditioners for symmetric and positive definite problems.

\section*{Acknowledgements}

This work was supported by the Austrian Science Fund (FWF) under the grant W1214, project DK4. This support is gratefully acknowledged.

\bibliographystyle{abbrv}

\bibliography{SpaceTime.bib}

\begin{thebibliography}{10}

\bibitem{HLN:AdamsFournier:2003a}
R.~A. Adams and J.~J.~F. Fournier.
\newblock {\em Sobolev spaces}, volume 140 of {\em Pure and Applied Mathematics
  (Amsterdam)}.
\newblock Elsevier/Academic Press, Amsterdam, second edition, 2003.

\bibitem{HLN:BazilevsCaloCottrellEvans:2010a}
Y.~Bazilevs, V.~Calo, J.~Cottrell, J.~Evans, T.~Hughes, S.~Lipton, M.~Scott,
  and T.~Sederberg.
\newblock Isogeometric analysis using {T}-splines.
\newblock {\em Computer Methods in Applied Mechanics and Engineering},
  199(5–8):229 -- 263, 2010.
\newblock Computational Geometry and Analysis.

\bibitem{HLN:BeiraodaVeigaBuffaSangalliVazquez:2014a}
L.~{Beir\~ao da Veiga}, A.~Buffa, G.~Sangalli, and R.~V\'{a}zquez.
\newblock Mathematical analysis of variational isogeometric methods.
\newblock {\em Acta Numerica}, 23:157--287, 2014.

\bibitem{HLN:BerghLofstrom:1976a}
J.~Bergh and J.~L\"ofstr\"om.
\newblock {\em Interpolation spaces. {A}n introduction}.
\newblock Springer-Verlag, Berlin-New York, 1976.
\newblock Grundlehren der Mathematischen Wissenschaften, No. 223.

\bibitem{HLN:CotrellHughesBazilevs:2009a}
J.~A. Cottrell, T.~J.~R. Hughes, and Y.~Bazilevs.
\newblock {\em Isogeometric {A}nalysis, Toward {I}ntegration of {CAD} and
  {FEA}}.
\newblock John Wiley and Sons, 2009.

\bibitem{HLN:Gander:2015a}
M.~Gander.
\newblock 50 years of time parallel time integration.
\newblock In T.~Carraro, M.~Geiger, S.~K\"orkel, and R.~Rannacher, editors,
  {\em Multiple Shooting and Time Domain Decomposition}, pages 69--114.
  Springer-Verlag, 2015.

\bibitem{HLN:GanderNeumueller:2016a}
M.~Gander and M.~Neum\"uller.
\newblock Analysis of a new space-time parallel multigrid algorithm for
  parabolic problems.
\newblock {\em SIAM J. Sci. Comput.}, 38(4):A2173--A2208, 2016.

\bibitem{HLN:GiannelliJuettlerSpeleers:2012a}
C.~Giannelli, B.~J\"uttler, and H.~Speleers.
\newblock {THB}-splines: the truncated basis for hierarchical splines.
\newblock {\em Computer Aided Geometric Design}, 29(7):485--498, 2012.

\bibitem{HLN:GiannelliJuettlerSpeleers:2014a}
C.~Giannelli, B.~J\"uttler, and H.~Speleers.
\newblock Strongly stable bases for adaptively refined multilevel spline
  spaces.
\newblock {\em Advances in Computational Mathematics}, 40:459--490, 2014.

\bibitem{HLN:Hofer:2018a}
C.~Hofer.
\newblock Analysis of discontinuous galerkin dual-primal isogeometric tearing
  and interconnecting methods.
\newblock {\em Mathematical Models and Methods in Applied Sciences},
  28(1):131--158, 2018.

\bibitem{HLN:HoferLanger:2019b}
C.~Hofer and U.~Langer.
\newblock Dual-primal isogeometric tearing and interconnecting methods.
\newblock In B.~Chetverushkin, W.~Fitzgibbon, Y.~Kuznetsov, P.~Neittanmakki,
  J.~Periaux, and O.~Pironneau, editors, {\em Contributions to Partial
  Differential Equations and Applications}, volume~47 of {\em Springer-ECCOMAS
  series ''Computational Methods in Applied Sciences''}. Springer, Berlin,
  Heidelberg, New York, 2016.
\newblock to appear.

\bibitem{HLN:HoferLanger:2017a}
C.~Hofer and U.~Langer.
\newblock Dual-primal isogeometric tearing and interconnecting solvers for
  multipatch {dG-IgA} equations.
\newblock {\em Comput. Methods Appl. Mech. Engrg.}, 316:2--21, 2017.

\bibitem{HLN:HoferLangerNeumuellerToulopoulos:2017a}
C.~Hofer, U.~Langer, M.~Neum\"uller, and I.~Toulopoulos.
\newblock Time-multipatch discontinuous {G}alerkin space-time isogeometric
  analysis of parabolic evolution problems.
\newblock RICAM Report 2017-26, Johann Radon Institute for Computational and
  Applied Mathematics, Linz, 2017.
\newblock available at
  https://www.ricam.oeaw.ac.at/files/reports/17/rep17-26.pdf.

\bibitem{HLN:HofreitherTakacs:2017a}
C.~Hofreither and S.~Takacs.
\newblock Robust multigrid for isogeometric analysis based on stable splittings
  of spline spaces.
\newblock {\em SIAM J. on Numerical Analysis}, 4(55):2004--2024, 2017.

\bibitem{HLN:HofreitherTakacsZulehner:2017a}
C.~Hofreither, S.~Takacs, and W.~Zulehner.
\newblock A robust multigrid method for isogeometric analysis in two dimensions
  using boundary correction.
\newblock {\em Computer Methods in Applied Mechanics and Engineering},
  316:22--42, 2017.

\bibitem{HLN:HughesBrooks:1982a}
T.~J.~R. Hughes and A.~Brooks.
\newblock Streamline upwind / {P}etrov-{G}alerkin formulation for convection
  dominated flows with particular emphasis on the incompressible navier-stokes
  equations.
\newblock {\em Comp. Meth. Appl. Mech. Engrg.}, 32:199--259, 1982.

\bibitem{HLN:HughesCottrellBazilevs:2005a}
T.~J.~R. Hughes, J.~A. Cottrell, and Y.~Bazilevs.
\newblock Isogeometric analysis: {CAD}, finite elements, {NURBS}, exact
  geometry and mesh refinement.
\newblock {\em Comput. Methods Appl. Mech. Engrg.}, 194:4135--4195, 2005.

\bibitem{HLN:JuettlerLangerMantzaflarisMooreZulehner:2014a}
B.~J\"uttler, U.~Langer, A.~Mantzaflaris, S.~E. Moore, and W.~Zulehner.
\newblock Geometry + {S}imulation {M}odules: {I}mplementing {I}sogeometric
  {A}nalysis.
\newblock In P.~Steinmann and G.~Leugering, editors, {\em PAMM}, volume~14 of
  {\em 1}, pages 961--962, Erlangen, 2014.

\bibitem{HLN:PARDISO500}
A.~Kuzmin, M.~Luisier, and O.~Schenk.
\newblock Fast methods for computing selected elements of the greens function
  in massively parallel nanoelectronic device simulations.
\newblock In F.~Wolf, B.~Mohr, and D.~Mey, editors, {\em Euro-Par 2013 Parallel
  Processing}, volume 8097 of {\em Lecture Notes in Computer Science}, pages
  533--544. Springer Berlin Heidelberg, 2013.

\bibitem{HLN:Ladyzhenskaya:1973a}
O.~A. Ladyzhenskaya.
\newblock {\em The Boundary Value Problems of Mathematical Physics}.
\newblock Nauka, Moscow, 1973.
\newblock In Russian. Translated in \emph{Appl. Math. Sci.} 49, Springer, 1985.

\bibitem{HLN:LadyzhenskayaSolonnikovUralceva:1967a}
O.~A. Ladyzhenskaya, V.~A. Solonnikov, and N.~N. Uraltseva.
\newblock {\em Linear and Quasilinear Equations of Parabolic Type}.
\newblock Nauka, Moscow, 1967.
\newblock In Russian. Translated in \emph{AMS}, Providence, RI, 1968.

\bibitem{HLN:Lang:2000a}
J.~Lang.
\newblock {\em Adaptive Multilevel Solution of Nonlinear Parabolic PDE Systems.
  Theory, Algorithm, and Applications}, volume~16 of {\em Lecture Notes in
  Computational Sciences and Engineering}.
\newblock Springer Verlag, Heidelberg, Berlin, 2000.

\bibitem{HLN:LangerMooreNeumueller:2016a}
U.~Langer, S.~Moore, and M.~Neum\"uller.
\newblock Space-time isogeometric analysis of parabolic evolution equations.
\newblock {\em Comput. Methods Appl. Mech. Engrg.}, 306:342--363, 2016.

\bibitem{gismoweb}
A.~Mantzaflaris, C.~Hofer, et~al.
\newblock {G+Smo (Geometry plus Simulation modules) v0.8.1}.
\newblock http://gs.jku.at/gismo, 2015.

\bibitem{HLN:Neumueller:2013a}
M.~Neum\"uller.
\newblock {\em Space-Time Methods: Fast Solvers and Applications}, volume~20 of
  {\em Monographic Series TU Graz: Computation in Engineering and Science}.
\newblock TU Graz, 2013.

\bibitem{HLN:SangalliTani:2016a}
G.~Sangalli and M.~Tani.
\newblock Isogeometric preconditioners based on fast solvers for the sylvester
  equation.
\newblock {\em SIAM Journal on Scientific Computing}, 38(6):A3644--A3671, 2016.

\bibitem{HLN:Steinbach:2015a}
O.~Steinbach.
\newblock Space-time finite element methods for parabolic problems.
\newblock {\em Comput. Meth. Appl. Math.}, 15(4):551--566, 2015.

\bibitem{HLN:Stynes:2005a}
M.~Stynes.
\newblock Steady--state convection--diffusion problems.
\newblock {\em Acta Numerica}, 14:445--508, 2005.

\bibitem{HLN:Tani:2017a}
M.~Tani.
\newblock A preconditioning strategy for linear systems arising from
  nonsymmetric schemes in isogeometric analysis.
\newblock {\em Computers \& Mathematics with Applications}, 74(7):1690 -- 1702,
  2017.

\bibitem{HLN:Thomme:2006a}
V.~Thome\'e.
\newblock {\em Galerkin finite element methods for parabolic problems},
  volume~25 of {\em Springer Series in Computational Mathematics}.
\newblock Springer-Verlag Berlin Heidelberg, 2006.

\bibitem{HLN:WarburtonHesthaven:2003}
T.~Warburton and J.~Hesthaven.
\newblock On the constants in hp-finite element trace inverse inequalities.
\newblock {\em Computer Methods in Applied Mechanics and Engineering},
  192(25):2765 -- 2773, 2003.

\bibitem{HLN:MonikaDiss}
M.~Wolfmayr.
\newblock {\em {Multiharmonic Finite Element Analysis of Parabolic
  Time-Periodic Simulation and Optimal Control Problems}}.
\newblock PhD thesis, Johannes Kepler University, Institute of Computational
  Mathematics, 2014.
\newblock available at
  http://www.numa.uni-linz.ac.at/Teaching/PhD/Finished/wolfmayr.

\bibitem{HLN:Zulehner:2011a}
W.~Zulehner.
\newblock Nonstandard norms and robust estimates for saddle point problems.
\newblock {\em SIAM Journal on Matrix Analysis and Applications},
  32(2):536--560, 2011.

\end{thebibliography}

\end{document}